\documentclass[11pt]{amsart}

\usepackage{amssymb,amsmath,amsthm,amsfonts,times,hyperref,mathrsfs,multirow}

\PassOptionsToPackage{hyphens}{url}
\usepackage{hyperref}

\newtheorem{theorem}{Theorem}[section]

\newtheorem{cor}[theorem]{Corollary}

\newtheorem{drconj}[theorem]{Dyson's Rank Conjecture}
\newtheorem{prop}[theorem]{Proposition}

\theoremstyle{definition}
\newtheorem{definition}[theorem]{Definition}

\theoremstyle{remark}
\newtheorem*{remark}{Remark}

\numberwithin{equation}{section}

\newcommand{\SL}{\mbox{SL}}

\newcommand{\SLZ}{\mbox{SL}_2(\mathbb{Z})}

\newcommand\Lpar[1]{\left(#1\right)}

\newcommand\Fell[4]{\mathcal{F}_{#1}\left(\frac{#2}{#3};#4\right)}
\newcommand\SFell[4]{\mathcal{F}_{#1}^{*}\left(\frac{#2}{#3};#4\right)}

\newcommand\SFTk{\SFell{}{1}{p}{\frac{z+k}{p}}}
\newcommand\SFTZ{\SFell{}{1}{p}{\frac{z}{p}}}

\newcommand\Gell[4]{\mathcal{G}_{#1}\left(\frac{#2}{#3};#4\right)}

\newcommand\JSpz[1]{\mathcal{J}^{*}\left(\frac{1}{p};{#1}\right)}

\newcommand\twidit[1]{\overset {\text{\lower 3pt\hbox{$\sim$}}}#1}
\newcommand\dtwidit[1]{\overset {\text{\lower 6pt\hbox{$\sim$}}}#1}
\newcommand\Wtwid{\overset {\text{\lower 3pt\hbox{$\sim$}}}W}
\newcommand\gtwid{\overset {\text{\lower 3pt\hbox{$\sim$}}}g}
\newcommand\ttwid{\overset {\text{\lower 3pt\hbox{$\sim$}}}\theta}
\newcommand\mutwid{\overset {\text{\lower 3pt\hbox{$\sim$}}}\mu}

\newcommand\AMat{\begin{pmatrix} a & b \\ c & d \end{pmatrix}}

\newcommand\SMat{\begin{pmatrix} 0 & -1 \\ 1 & 0 \end{pmatrix}}
\newcommand\TMat{\begin{pmatrix} 1 & 1 \\ 0 & 1 \end{pmatrix}}

\newcommand\FL[1]{\left\lfloor#1\right\rfloor}
\newcommand\CL[1]{\left\lceil#1\right\rceil}

\newcommand\strokeb[2]{\,\left\arrowvert\,\left[#1\right]_#2\right.}
\newcommand\stroke[3]{#1\,\left\arrowvert\,\left[#2\right]_{#3}\right.}

\newcommand\ord{\mbox{ord}}         
\newcommand\ORD{\mbox{ORD}}

\newcommand\gA{A}   

\newcommand\Cpz[1]{\mathcal{C}_p\left(#1\right)}
\allowdisplaybreaks

\newcommand\thm[1]{\ref{thm:#1}}

\newcommand\corol[1]{\ref{cor:#1}}

\newcommand\propo[1]{\ref{propo:#1}}
\newcommand\eqn[1]{(\ref{eq:#1})}
\newcommand\sect[1]{\ref{sec:#1}}
\newcommand\refdef[1]{\ref{def:#1}}

\begin{document}
\newcommand{\beqs}{\begin{equation*}}
\newcommand{\eeqs}{\end{equation*}}
\newcommand{\beq}{\begin{equation}}
\newcommand{\eeq}{\end{equation}}
\title[Crank function symmetries]
{Transformation and symmetries for the Andrews-Garvan crank function}

\author{Rishabh Sarma}
\address{Department of Mathematics, University of Florida, Gainesville,
FL 32611-8105}

\email{rishabh.sarma@ufl.edu}


\date{\today}


\begin{abstract}
Let $R(z,q)$ be  the two-variable generating function of Dyson's rank function. In a recent joint work with Frank Garvan, we investigated the transformation of the elements of the $p$-dissection of $R(\zeta_p,q)$, where $\zeta_p$ is a primitive $p$-th root of unity, under special congruence subgroups of $SL_2(\mathbb{Z})$, leading us to interesting symmetry observations. In this work, we derive analogous transformation results for the two-variable crank generating function $C(z,q)$ in terms of generalized eta products. We consider the action of the group $\Gamma_0(p)$ on the elements of the $p$-dissection of $C(\zeta_p,q)$, leading us to new symmetries for the crank function. As an application, we give a new proof of the crank theorem predicted by Dyson in 1944 and resolved by Andrews and Garvan in 1988. Furthermore, we present identities expressing the elements of the crank dissection in terms of generalized eta products for primes $p=11,13,17$ and $19$.
\end{abstract}

\maketitle
\section{Introduction}
\label{sec:intro}

The rank function is an important statistic in the theory of partitions. It is defined as the largest part minus the number of parts in a partition. In 1944, Dyson gave this definition of the rank statistic and conjectured that it decomposes the partitions of $5n+4$ and $7n+5$ into five and seven equinumerous classes respectively, a conjecture that was proven 10 years later by Atkin and Swinnerton-Dyer. He also prophesied the existence of a ``crank" statistic for a similar explanation of the mod $11$ congruence of Ramanujan viz. $$p(11n+6) \equiv 0 \pmod{11},$$ and it wasn't until 1988 that the elusive crank was found by Andrews and Garvan \cite{An-Ga88}. The crank of a partition is defined as the largest part if there are no ones in the partition and otherwise the number of parts larger than the number of ones minus the number of ones. It is worth mentioning here that after the discovery of the partition crank, Garvan, Kim and Stanton \cite{Ga-Ki-St} came up with new statistics on partitions (also called cranks) which combinatorially prove Ramanujan's congruences for the partition function modulo $5$, $7$, $11$ and $25$, giving explicit bijections between equinumerous crank classes. However, the crank statistic that we work on in this paper is the ordinary partition crank defined by Andrews and Garvan which we simply refer to as the crank. Let $M(r,t,n)$ denote the number of partitions of $n$ with crank congruent to $r$ mod $t$. Then they show that 
\begin{align}
M(r,5,5n+4)&=\frac{1}{5} p(5n+4),\, 0\leq r \leq 4, \label{eq:crank5}
\\
M(r,7,7n+5)&=\frac{1}{7} p(7n+5),\, 0\leq r \leq 6, \label{eq:crank7}
\\
M(r,11,11n+6)&=\frac{1}{11} p(11n+6),\, 0\leq r \leq 10, \label{eq:crank11}
\end{align}
thus providing an almost combinatorial explanation of all three partition congruences of Ramanujan. Let $q=\exp(2\pi i \tau)$, where $\tau \in \mathbb{H}$. Let $M(m,n)$ denote the number of partitions of $n$ with crank $m$. Then the two variable generating function for the crank function, due to Andrews and Garvan (see \cite{An-Ga88}) is given by 

$$\sum_{n=0}^\infty \sum_{m\in \mathbb{Z}} M(m,n)\,z^m\,q^n = (1-z)q + \dfrac{(q;q)_{\infty}}{(zq;q)_{\infty}(z^{-1}q;q)_{\infty}}.$$
\\
Equivalently we define

$$C(z,q) :=\dfrac{(q;q)_{\infty}}{(zq;q)_{\infty}(z^{-1}q;q)_{\infty}}=1+(-1+z+z^{-1})q+\sum_{n=2}^\infty \sum_{m\in \mathbb{Z}} M(m,n)\,z^m\,q^n.$$

In this paper, we look at the transformation and symmetry of this function, when $z$ is a primitive $p$-th root of unity $\zeta_p$. We find that the crank function satisfies the same type of symmetries as the rank generating function, which was our subject of interest in \cite{Ga-Ri}.
\\\\
Throughout this paper we use the standard $q$-notation:
\begin{align*}
(a;q)_{\infty} &= \prod_{k=0}^\infty (1-aq^k),
\\
(a;q)_{n} &= \frac{(a;q)_{\infty}}{(aq^n;q)_{\infty}},
\\
(a_1,a_2,\dots,a_j;q)_{\infty} 
&= (a_1;q)_{\infty}(a_2;q)_{\infty}\dots(a_j;q)_{\infty}.
\end{align*}

We also define a class of generalized eta products and functions that we will use frequently in expressions for our crank identities and to study their transformations in the subsequent sections.

\begin{definition}
\label{def:geneta}
Following Biagioli (see \cite{Bi89}), define
\beq
f_{N,\rho}(z) := (-1)^{\FL{\rho/N}}q^{(N-2\rho)^2/(8N)}\,(q^\rho,q^{N-\rho},q^N;q^N)_\infty,
\label{eq:fdef}
\eeq
\noindent
where $N, \rho \in \mathbb{Z}$ with $N \geq 1$ and $N \nmid \rho$. 
\\\\
Then, for a vector $\overrightarrow{n}=(n_0,n_1,n_2,\cdots ,n_{\frac{p-1}{2}}) \in \mathbb{Z}^{\frac{1}{2}(p+1)}$, define 
\beq
j(z)=j(p,\overrightarrow{n},z)=\eta(pz)^{n_0}\prod\limits_{k=1}^{\frac{p-1}{2}}f_{p,k}(z)^{n_k},
\label{eq:jdef}
\eeq 
\end{definition}

We also define the weight $k$ slash operator on the set of meromorphic functions on $\mathbb{H}$ by $$f(z)|_k(\gamma):=(cz+d)^{-k}f(\gamma z)$$ for all $\gamma = \AMat \in \SL_2(\mathbb{Z})$ and $z \in \mathbb{H}$. We remind the reader that the automorphy factor $(cz+d)^k$ appears as a prefactor in the transformation property when $f$ is weakly modular of weight $k$.
\\\\
Let $N(m,n)$ denote the number of partitions of $n$ with rank $m$ and let $R(z,q)$ denote
the two-variable generating function for the Dyson rank function so that
$$
R(z,q) = \sum_{n=0}^\infty \sum_{m \in \mathbb{Z}} N(m,n)\,z^m\,q^n.
$$

The modularity of the rank generating function $R(z,q)$, where $z$ is replaced by a primitive $p$-th root of unity for a general prime $p$ was first studied by Bringmann and Ono \cite{Br-On10}. Building on the groundbreaking work of Zwegers \cite{Zw00} and \cite{Zw-thesis} around the turn of the century, who realised how Ramanujan's mock theta functions occur naturally as the holomorphic parts of certain real analytic modular forms, Bringmann and Ono showed that $q^{-1/24}R(\zeta_p,q)$ is the holomorphic part a weak Maass form of weight $\frac{1}{2}$ on $\Gamma_1(576p^4)$, where the non-holomorphic part is a Mordell integral of a theta function. Garvan \cite{Ga19}, in 2019, observed that the introduction of a simple multiplier and a generalized correction series factor reduces the modularity of the rank function to larger and more natural congruence subgroups of $\SL_2(\mathbb{Z})$.
\\\\
In a subsequent joint work with Garvan \cite{Ga-Ri}, we generalized and improved these results of his paper \cite{Ga19} on transformations for Dyson's rank function leading us to exciting observations of symmetry among the elements of dissection of the rank generating function. In this paper, we try to implement our ideas in \cite{Ga-Ri} and those of Garvan \cite{Ga19} to obtain analogous results for the crank generating function.
\\\\\
Since its discovery, it has been well established that the crank statistic plays a central role in the theory of partitions. Perhaps the most important breakthrough came in 2005 when Mahlburg showed that the crank itself satisfies Ramanujan-type congruences using Hecke's theory of modular forms. In his seminal paper, Mahlburg establishes congruences for infinitely many non-nested arithmetic progressions modulo powers of primes by tying the crank generating function to the Dedekind eta function and Klein forms. Rewriting his functions in the framework of generalized eta products, we have the following result.
\\
\begin{theorem}[{Mahlburg \cite[Corollary 3.3]{Mah}}]
\label{thm:transthm}
Let $p>3$ be prime, suppose $1\le\ell\le(p-1)$, and define
\begin{equation*}
\Fell{}{\ell}{p}{z} := \dfrac{\eta(z)^2}{E_{0,\ell}(z)},
\end{equation*}
where, following Yang \cite{Yang04}, we consider the generalized eta product $$E_{0,\ell}(z)=q^\frac{1}{12}\prod\limits_{m=1}^{\infty}(1-\zeta_p^{\ell}q^{m-1})(1-\zeta_p^{-\ell}q^{m}).$$
Then $\Fell{}{\ell}{p}{z}$ is a weakly holomorphic modular form of weight 1 on the congruence subgroup $\Gamma_1(2p^2)$.
\end{theorem}

Mahlburg further goes on to establish a rescaled form of the crank generating function as the sum of a weakly holomorphic function on $\Gamma_1(2p^2)$ and a modular form on $\Gamma_0(p)$ of integral weights (\cite{Mah}, Section 4). We show that after multiplying with an appropriate eta product and a power of $q$, we can deduce the modularity of the crank generating function to a simpler congruence subgroup of $\SL_2(\mathbb{Z})$, resembling the result for the rank generating function \cite[Theorem 1.2]{Ga19}.

\begin{theorem}
\label{thm:modularthmintro}
Let $p > 3$ be prime. Then the function 
$$
\eta(p^2 z) \, \Cpz{z} 
$$
is a weakly holomorphic modular form of weight $1$ on the group $\Gamma_0(p^2) \cap \Gamma_1(p)$, where $\Cpz{z}=q^{\frac{-1}{24}}C(\zeta_p^{\ell},q)$.
\end{theorem}

This result is analogous to the modularity of the rank generating function that was studied in \cite{Ga19}. There the author also considered the modularity and transformation of elements of the $p$-dissection of the rank generating function. We note that in the definition for these rank dissection elements $\mathcal{K}_{p,m}(\zeta_p^d;z)$, there are two cases pertaining to whether $-24m$ is a quadratic residue or non-residue modulo $p$ {\cite[Definition 1.4]{Ga-Ri}}. In case of the crank generating function, we have the following analogous expression for elements of the dissection and subsequent modularity results. These dissection elements do not feature the correction factor involving the series $\Phi_{p,a}(q)$ as in the definition of the rank dissection elements $\mathcal{K}_{p,m}(\zeta_p^d;z)$ where the series $\Phi_{p,a}(q)$ appears in the case when $-24m$ is a quadratic residue modulo $p$ {\cite[Definition 1.4 (ii)]{Ga-Ri}}.
\begin{definition} 
\label{def:KCpm}
For $p>3$ prime, $s_p=\frac{1}{24}(p^2-1)$, $0 \le m \le p-1$ and $1 \le \ell \le p-1$, define
\begin{equation*}
\mathcal{S}_{p,m}^{}(\zeta_p^{\ell},z) := q^{m/p}\,\prod_{n=1}^\infty (1-q^{pn})\, 
\sum_{n=\CL{\frac{1}{p}(s_p-m)}}^\infty \left(\sum_{k=0}^{p-1} M(k,p,pn + m -s_p)\,\zeta_p^{k \ell}\right)q^n.
\end{equation*}
\label{eq:KCpm}
\end{definition}

\begin{theorem}
\label{thm:Kpmthmintro}
Let $p>3$ be prime and suppose $0 \le m \le p-1$. Then 
\begin{enumerate}
\item[(i)]
$\mathcal{S}_{p,0}^{}(\zeta_p,z)$ is a weakly holomorphic modular form of weight $1$ on
$\Gamma_1(p)$.
\item[(ii)]
If $1 \le m \le (p-1)$ then $\mathcal{S}_{p,m}(\zeta_p,z)$ is a weakly holomorphic modular form of weight $1$ on
$\Gamma(p)$.                           
In particular,
$$
\stroke{\mathcal{S}_{p,m}(\zeta_p,z)}{A}{1}
= \exp\Lpar{\frac{2\pi ibm}{p}}\,\mathcal{S}_{p,m}(\zeta_p,z),
$$
for $A=\AMat\in\Gamma_1(p)$.                       
\end{enumerate}
\end{theorem}

Analogous to the rank, the action of the congruence subgroup $\Gamma_0(p)$ on $\mathcal{S}_{p,m}(\zeta_p,z)$ leads to observation of symmetry among the elements of the crank dissection. 

\begin{theorem}
\label{thm:symmthm1intro}
Let $p>3$ be prime, $0 \le m \le p-1$, and $1 \le d \le p-1$.             
Then\\
\beq
\stroke{\mathcal{S}_{p,m}(\zeta_p, z)}{A}{1}
= \frac{\sin(\pi/p)}{\sin(d\pi/p)} \, (-1)^{d+1}\,\zeta_p^{mak}
\,\mathcal{S}_{p,{ma^2}}^{}(\zeta_p^d,z),
\label{eq:KpmAtrans}
\eeq
assuming $1 \le a,d \le (p-1)$ 
and
$$
A = \begin{pmatrix}
      a & k \\ p & d 
    \end{pmatrix}
\in \Gamma_0(p).
$$
\end{theorem}

Note : When $m$ is a quadratic residue/non-residue modulo $p$, the same is correspondingly true for ${ma^2}$.
\\\\
The $m=0$ element of the dissection viz. ${\mathcal{S}_{p,0}^{}(\zeta_p, z)}$ defined above is also of further special interest. A beautiful symmetry among the zeta-coefficients in the identity for ${\mathcal{K}_{p,0}^{}(\zeta_p, z)}$ element of the rank dissection expressed in terms of generalized eta functions $j(p,\overrightarrow{n},z)$ was given in \cite{Ga-Ri}. Using our main symmetry result in Theorem \thm{symmthm1intro} above and the conditions for modularity of the generalized eta functions \cite[Theorem 3.1]{Ga-Ri}, we can determine the transformation of $\mathcal{S}_{p,0}^{}(\zeta_p, z)$ and find that the exact same symmetry result holds for this crank counterpart of the $m=0$ element. We do not pursue this here,  but the interested reader could investigate further (see \cite{Ga-Ri}, Theorem 5.1 for details).
\\\\
The paper is organized as follows. In Section \sect{prelims} we build up the framework of our results and introduce a class of eta functions arising in the definition of our completed crank function and determine their transformation under special congruence subgroups of $\SL_2(\mathbb{Z})$. In Section \sect{modsymmres}, we prove our main results on the modularity, transformation and symmetry of the crank function. Section \sect{cuspords} is devoted to calculating lower bounds for the orders of $\mathcal{S}_{p,m}(\zeta_{p},z)$ at the cusps of $\Gamma_1(p)$ which we utilize to prove identities in the subsequent section. Finally, in Section \sect{crankids}, we give a new proof of the crank theorem and also state identities for $\mathcal{S}_{p,m}(\zeta_{p},z)$ when $p=13,17$ and $19$ in terms of generalized eta-functions, proving two such identities mod $13$ using the Valence formula, the modularity conditions and the orders at cusps determined in the previous section.\\

\section{Preliminary definitions and Results}
\label{sec:prelims}

In Section \sect{intro}, we defined $M(r,t,n)$ to be the number of partitions of $n$ with crank
congruent to $r$ mod $t$, and let $\zeta_p=\exp(2\pi i/p)$. Then we have

$$C(\zeta_p,q) = \sum_{n=0}^\infty \Lpar{\sum_{k=0}^{p-1} M(k,p,n)\,\zeta_p^k}\,q^n.$$
\\
We introduce a class of generalized eta products and its transformation under matrices in $\SL_2(\mathbb{Z})$ due to Yifan Yang. We later show how it relates to the crank function.

\begin{theorem}[Yang \cite{Yang04}] 
\label{thm:genetaYang}
Let $p$ be prime and $g,h$ be arbitrary real numbers not simultaneously congruent to $0$ modulo $p$. For $z \in \mathbb{H}$, we define the generalized Dedekind eta functions $E_{g,h}(z)$ by
$$E_{g,h}(z)=q^{B(g/p)/2}\prod\limits_{m=1}^{\infty}\left(1-\zeta_p^hq^{m-1+\frac{g}{p}}\right)\left(1-\zeta_p^{-h}q^{m-\frac{g}{p}}\right),$$
where $B(x)=x^2-x+1/6$. Then the following transformations hold :
\begin{enumerate}
\item[i)]   $E_{g+p,h}=E_{-g,-h}=-\zeta_p^{-h}E_{g,h}$.
\item[ii)]  $E_{g,h+p}=E_{g,h}$.
\\
Moreover for $\gamma=\AMat \in \SL_2(\mathbb{Z})$, we have
\item[iii)] $E_{g,h}(z+b)=e^{\pi i b B(g/p)}E_{g,bg+h}(z)$ for $c=0$.
\item[iv)]  $E_{g,h}(\gamma z)=\varepsilon(a,b,c,d)e^{\pi i \delta}E_{g',h'}(z)$ for $c\neq 0$ where 
\\
\begin{align*}
\varepsilon(a,b,c,d)&=\begin{cases}
e^{\pi i (bd(1-c^2)+c(a+d-3))/6} & \mbox{if $c$ is odd},
\\\\
-i e^{\pi i (ac(1-d^2)+d(b-c+3))/6} & \mbox{if $d$ is odd},
\end{cases}
\\\\
\delta&=\frac{g^2ab+2ghbc+h^2cd}{p^2}-\frac{gb+h(d-1)}{p},~and
\\\\
(g'\,\,\,\,h')&=(g\,\,\,\,h)\AMat.
\end{align*}
\end{enumerate}
\end{theorem}

We now use Yang's transformation results to find the transformation of our function $\Fell{}{\ell}{p}{z}$ defined below under $\Gamma_0(p)$.

\begin{prop}
\label{propo:F1trans1}
Let $p>3$ be prime, suppose $1\le\ell\le(p-1)$, and define
\begin{equation*}
\Fell{}{\ell}{p}{z} := \dfrac{\eta(z)^2}{E_{0,\ell}(z)},
\end{equation*}
where, following Yang \cite{Yang04}, we consider the generalized eta product $$E_{0,\ell}(z)=q^\frac{1}{12}\prod\limits_{m=1}^{\infty}(1-\zeta_p^{\ell}q^{m-1})(1-\zeta_p^{-\ell}q^{m}).$$
Then
\begin{equation*}
\stroke{\Fell{}{\ell}{p}{z}}{A}{1} = \mu(A,\ell)\,
\Fell{}{\overline{d\ell}}{p}{z},
\end{equation*}
where
$$\mu(A,\ell) = \exp\left(\pi i \left(\frac{\ell(d-1)}{p}+\frac{c d \ell^2}{p^2}-\frac{c \ell}{p}\right)\right) ,~\text{and}~ A = \AMat \in \Gamma_0(p).$$
Here $\overline{m}$ is the least nonnegative residue of $m\pmod{p}$.
\end{prop}

\begin{proof}
Using Theorems \thm{genetaYang} and {\cite[Theorem 2, p.51]{Kn-book}}, we have
\\
\begin{align*}
\stroke{\Fell{}{\ell}{p}{z}}{A}{1} &=\quad
\begin{cases}
\dfrac{e^{\frac{2\pi i b}{12}}\eta(z)^2}{e^{\frac{\pi i b}{6}}E_{0,\ell}(z)} & \mbox{if $c=0$},
\\\\
 \dfrac{e^{\frac{\pi i}{6}\left[(a+d)c-bd(c^2-1)-3c\right]}\eta(z)^2}{e^{\frac{\pi i}{6}\left[bd(1-c^2)+c(a+d-3)\right]} e^{\pi i\left[\frac{\ell^2cd}{p^2}-\frac{\ell(d-1)}{p}\right]}E_{\ell c,\ell d}(z)} & \mbox{if $c\neq 0$ and odd},
\\\\
\dfrac{e^{\frac{\pi i}{6}\left[(a+d)c-bd(c^2-1)-3c\right]}\eta(z)^2}{-i e^{\frac{\pi i}{6}\left[bd(1-c^2)+c(a+d-3)\right]} e^{\pi i\left[\frac{\ell^2cd}{p^2}-\frac{\ell(d-1)}{p}\right]}E_{\ell c,\ell d}(z)} & \mbox{if $c\neq 0$ and even},
\end{cases}
\\\\
&=e^{\pi i\left[\frac{\ell(d-1)}{p}-\frac{\ell^2cd}{p^2}\right]}\dfrac{\eta(z)^2}{E_{\ell c,\ell d}(z)}.
\\
\end{align*}
Now, using Theorem \thm{genetaYang} we have $E_{\ell c,\ell d}(z)=\left(-\zeta_p^{-\ell d}\right)^{\frac{\ell c}{p}}E_{0,\ell d}(z)=e^{\pi i\left[\frac{c\ell}{p}-\frac{2\ell^2 c d}{p^2}\right]}E_{0,\overline{d\ell}}(z)$.
\\\\
Then, $$\stroke{\Fell{}{\ell}{p}{z}}{A}{1} = e^{\pi i \left[\frac{\ell(d-1)}{p}+\frac{c d \ell^2}{p^2}-\frac{c \ell}{p}\right]}\dfrac{\eta(z)^2}{E_{0,\overline{d \ell}}(z)} $$
\\
and we have our result.
\end{proof}

The following corollary now follows easily and gives a simpler transformation.

\begin{cor} 
\label{cor:F1trans2}
Let $p>3$ be prime and suppose $1\leq \ell \leq \frac{1}{2}(p-1)$. Then
\begin{equation*}
\stroke{\Fell{}{\ell}{p}{z}}{A}{1} = \Fell{}{\ell}{p}{z},
\end{equation*}
where $A \in \Gamma_0(p^2)\cap \Gamma_1(p).$
\end{cor}

\begin{proof}
Let $A = \AMat \in \Gamma_0(p^2) \cap \Gamma_1(p).$ Then $c\equiv0\pmod{p^2}$ and $a\equiv d\equiv1\pmod{p}$. So 
\\
\begin{enumerate}
\item[i)] $p\ell(d-1)+cd\ell^2-c\ell p\equiv 0 \pmod {p^2},$
\item[ii)] $p\ell(d-1)+cd\ell^2-c\ell p\equiv d\ell+\ell+cd\ell+c\ell \equiv \ell(c+1)(d+1) \equiv 0 \pmod {2}.$
\\\\
\end{enumerate}
Therefore, $$\mu(A,\ell) = \exp\left(\pi i \left(\frac{\ell(d-1)}{p}+\frac{c d \ell^2}{p^2}-\frac{c \ell}{p}\right)\right)=1$$
\\
and we have our result.
\end{proof}

\section{Proofs of main results}
\label{sec:modsymmres}

\subsection{Modularity and transformation results}\hfill\\

Let $\SFell{}{\ell}{p}{z}=\dfrac{\eta(p^2z)}{\eta(z)}\,\Fell{}{\ell}{p}{z}$. The generating function of the crank due to Andrews and Garvan is given by 
\begin{align*}
C(z,q) &= \sum_{n=0}^\infty \sum_m M(m,n)\,z^m\,q^n=\prod\limits_{n=1}^{\infty}\dfrac{1-q^n}{(1-zq^n)(1-z^{-1}q^n)}.
\end{align*} 
Replacing $z$ by $\zeta_p^{\ell}$ in the above expression and multiplying both sides by $\eta(p^2z)$, we can express the crank generating function in terms of our eta functions as 
\beq
\eta(p^2z)\, \Cpz{z}=(1-\zeta_p^{\ell})\,\SFell{}{\ell}{p}{z},~\text{where}~\Cpz{z}=q^{-\frac{1}{24}}C(\zeta_p^{\ell},q).
\label{eq:cranketaeq}
\eeq

This leads to our first main result stated in the introduction section. We restate it here.
\\
\begin{theorem}
\label{thm:modularthm}
Let $p > 3$ be prime. Then the function 
$$
\eta(p^2 z) \, \Cpz{z} 
$$
is a weakly holomorphic modular form of weight $1$ on the group $\Gamma_0(p^2) \cap \Gamma_1(p)$.
\end{theorem}

\begin{proof}
In the light of Equation \eqn{cranketaeq}, the proof of the theorem is equivalent to showing that $\SFell{}{\ell}{p}{z}$ is a weakly holomorphic modular form of weight $1$ on the group $\Gamma_0(p^2) \cap \Gamma_1(p)$. The transformation condition holds using Corollary \corol{F1trans2} and the well-known result that $\dfrac{\eta(p^2z)}{\eta(z)}$ is a modular function on $\Gamma_0(p^2)$ when $p>3$ is prime. That is 
\begin{equation*}
\stroke{\SFell{}{\ell}{p}{z}}{A}{1} = \SFell{}{\ell}{p}{z}
\end{equation*}
for $A \in \Gamma_0(p^2) \cap \Gamma_1(p)$.
\\\\
We next check the cusp conditions and show that $\stroke{\SFell{}{\ell}{p}{z}}{A}{1}(z)$ when expanded as a series in $q_{p^2}=\exp(\frac{2\pi i z}{p^2})$ has only finitely many terms with negative exponents for $A \in \SL_2(\mathbb{Z})$. By Theorems \thm{genetaYang} and {\cite[Theorem 2, p.51]{Kn-book}}, we have
$$\stroke{\SFell{}{\ell}{p}{z}}{A}{1}=\stroke{\dfrac{\eta(p^2z)}{\eta(z)}\dfrac{\eta(z)^2}{E_{0,\ell}(z)}}{A}{1}=\stroke{\dfrac{\eta(p^2z)}{\eta(z)}}{A}{1}.\varepsilon_A\dfrac{\eta(z)^2}{E_{\ell c,\ell d}(z)},$$
where $\varepsilon_A$ is a root of unity. Using the fact that $\dfrac{\eta(p^2z)}{\eta(z)}$ is a modular function on $\Gamma_0(p^2)$ and the definition of $\eta(z)$ and $E_{g,h}(z)$, we find that the series has only finitely many terms with negative exponents. \\
\end{proof}

Next, we define the (weight $k$) Atkin $U_p$ operator by
\beq
\stroke{F}{U_p}{k} := \frac{1}{{p}} \sum_{r=0}^{p-1} F\Lpar{\frac{z+r}{p}} 
= p^{\frac{k}{2}-1}\sum_{n=0}^{p-1} \stroke{F}{T_r}{k},
\label{eq:Updef}
\eeq
where
$$
T_r = \begin{pmatrix} 1 & r \\ 0 & p \end{pmatrix},
$$
and the more general $U_{p,m}$ is defined as
\beq
\stroke{F}{U_{p,m}}{k} := \frac{1}{p} \sum_{r=0}^{p-1} 
\exp\Lpar{-\frac{2\pi irm}{p}}\,F\Lpar{\frac{z+r}{p}} 
= p^{\frac{k}{2}-1} 
\sum_{r=0}^{p-1} \exp\Lpar{-\frac{2\pi irm}{p}}\, \stroke{F}{T_r}{k}. 
\label{Upkdef}
\eeq
We note that $U_p = U_{p,0}$. In addition, if
$$
F(z) = \sum_n a(n) q^n = \sum_n a(n)\,\exp(2\pi izn),
$$
then
$$
\stroke{F}{U_{p,m}}{k} = q^{m/p} \sum_n a(pn+m)\,q^n = \exp(2\pi imz/p)\, 
\sum_n a(pn+m)\,\exp(2\pi inz).
$$

Combining Equation \eqn{cranketaeq} and Definition \refdef{KCpm}, we have
\begin{prop}
\label{propo:KCpm2}
For $p>3$ be a prime and $0 \le m \le p-1$ we have
\beq
\mathcal{S}_{p,m}(\zeta_p^{\ell},z)=(1-\zeta_p^{\ell})\,\stroke{\SFell{}{\ell}{p}{z}}{U_{p,m}}{1}
\label{eq:KCpmdef}
\eeq
\end{prop} 

The modularity result of the crank function stated in the introduction section follows easily now. We restate the theorem here and note that the proof utilizes the same technique as it did for the rank dissection elements $\mathcal{K}_{p,m}^{}(z)$ {\cite[Theorem 6.3, p.234]{Ga19}}. The functions involved in the transformation are however different.
\begin{theorem}
\label{thm:Kpmthm}
Let $p>3$ be prime and suppose $0 \le m \le p-1$. Then 
\begin{enumerate}
\item[(i)]
$\mathcal{S}_{p,0}^{}(z)$ is a weakly holomorphic modular form of weight $1$ on
$\Gamma_1(p)$.
\item[(ii)]
If $1 \le m \le (p-1)$ then $\mathcal{S}_{p,m}(z)$ is a weakly holomorphic modular form of weight $1$ on
$\Gamma(p)$.                           
In particular,
$$
\stroke{\mathcal{S}_{p,m}(z)}{A}{1}
= \exp\Lpar{\frac{2\pi ibm}{p}}\,\mathcal{S}_{p,m}(z),
$$
for $A=\AMat\in\Gamma_1(p)$.                       
\end{enumerate}
\end{theorem}

\begin{proof}
We let
$$
A = \AMat \in \Gamma_1(p),
$$
and undergo the same matrix transformations as in the proof of {\cite[Theorem 6.3, p.234]{Ga19}}, with our function $\SFell{}{1}{p}{z}$ replacing $\JSpz{z}{}$, and apply Corollary \corol{F1trans2} (since $B_k\in \Gamma_0(p^2) \cap \Gamma_1(p)$) to arrive at 
\begin{align*}
\stroke{\mathcal{S}_{p,m}(z)}{A}{1}
&= \frac{1}{\sqrt{p}}(1-\zeta_p^{})\, \exp\Lpar{\frac{2\pi ibm}{p}}\,\sum_{k'=0}^{p-1} \exp\Lpar{-\frac{2\pi ik'm}{p}}\,\stroke{\SFell{}{1}{p}{z}}{T_{k'}}{1} 
\\
&=\exp\Lpar{\frac{2\pi ibm}{p}}\,\mathcal{S}_{p,m}(z).
\end{align*} 
It is easy to see that each $\mathcal{S}_{p,m}(z)$ is holomorphic on $\mathcal{H}$. The cusp conditions follow by a standard argument. We examine orders at each cusp in more detail in a later section.
\end{proof}

\subsection{Symmetry result analogous to the rank}

\begin{theorem}
\label{thm:symmthm1}
Suppose $p>3$ prime, $0 \le m \le p-1$, and $1 \le d \le p-1$.             
Then\\
\beq
\stroke{\mathcal{S}_{p,m}(\zeta_p, z)}{A}{1}
=\frac{\sin(\pi/p)}{\sin(d\pi/p)} (-1)^{d+1}\,\zeta_p^{mak}\,\mathcal{S}_{p,{ma^2}}^{}(\zeta_p^d,z),
\label{eq:KpmAtrans}
\eeq
assuming $1 \le a,d \le (p-1)$ 
and
$$
A = \begin{pmatrix}
      a & k \\ p & d 
    \end{pmatrix}
\in \Gamma_0(p).
$$
\end{theorem}

\begin{proof}
Once again, we undergo the same matrix transformations as we did in the case of ${\mathcal{K}_{p,m}^{}(\zeta_p,z)}$ {\cite[Proposition 4.7]{Ga-Ri}}, with our function $\Fell{}{1}{p}{z}$ replacing $\Fell{1}{1}{p}{z}$, and apply Proposition \propo{F1trans1}. Since $B_r\in \Gamma_0(p) $, we have 
\begin{align*}
\mu(B_r,1) &= \exp\left(\pi i \left(\frac{(d-r'p-1)}{p}+\frac{p^2(d-r'p)}{p^2}-p\right)\right)
\\
&=(-1)^{d+1}\,\exp\left(\pi i \left(\frac{d-1}{p}\right)\right)(-1)^{-r'(p+1)},
\\
&=(-1)^{d+1}\,\zeta_p^{\frac{d-1}{2}}.
\end{align*}
Therefore
\begin{align*}
\mathcal{S}_{p,m}(\zeta_p, z) \strokeb{A}{1} & =(1-\zeta_p^{})\, \SFell{}{1}{p}{z} 
  \strokeb{U_{p,m}}{1} \strokeb{A}{1}  
\\
& = \frac{1-\zeta_p}{\sqrt{p}} \sum_{r=0}^{p-1}
  \zeta_p^{-rm} 
  \frac{\eta(p^2 z)}{\eta(z)} \, 
   \mu(B_r,1) 
   \Fell{}{d}{p}{z}
   \strokeb{T_{r'}}{1} 
\\
& = \frac{\sin(\pi/p)}{\sin(d\pi/p)} (-1)^{d} \zeta_p^{mak} \frac{(1-\zeta_p^d)}{\sqrt{p}} \sum_{r'=0}^{p-1}
  \zeta_p^{-r'ma^2} 
  \frac{\eta(p^2 z)}{\eta(z)} \, 
   \Fell{}{d}{p}{z}
   \strokeb{T_{r'}}{1},   
\end{align*}
using the facts that
\begin{align*}
\frac{(1-\zeta_p)\,\zeta_p^{\frac{d-1}{2}}}{(1-\zeta_p^d)}= \frac{\sin(\pi/p)}{\sin(d\pi/p)},
\\
\zeta_p^{-rm} = \zeta_p^{m(-r'a^2 +ak)}= \zeta_p^{mak} \zeta_p^{-mr'a^2},
\end{align*}
and as $r$ runs through a complete residue system mod $p$ so does $r'$.
The result follows.
\end{proof}

We end this section with an illustration of the validity and advantage of our theorem. The following are the elements of the $11$-dissection of  $\eta(p^2z) \, \Cpz{\zeta_p, z}$ for $p=11$ in terms of generalized eta-functions. A similar form of this identity can also be found in the works of Garvan \cite{Ga90} and Hirschhorn \cite{Hi90}.

\begin{align*}
\mathcal{S}_{11,0}(\zeta_{11}, z)&=0,
\\
\mathcal{S}_{11,{1}}(\zeta_{11}, z)&=- (1+\zeta_{11}^2+\zeta_{11}^2 + \zeta_{11}^5+\zeta_{11}^6+\zeta_{11}^8+\zeta_{11}^9) \, \frac{\eta(11z)^2}{\eta_{11,{3}}(z)},
\\
\mathcal{S}_{11,{2}}(\zeta_{11}, z)&=- (1+\zeta_{11}^2+ \zeta_{11}^5+\zeta_{11}^6+\zeta_{11}^9) \frac{\eta(11z)^2 \, \eta_{11,{1}}(z)}{\eta_{11,{4}}(z)\, \eta_{11,{5}}(z)},
\\
\mathcal{S}_{11,{3}}(\zeta_{11}, z)&=- (1+\zeta_{11}^4+\zeta_{11}^7) \, \frac{\eta(11z)^2}{\eta_{11,{4}}(z)},
\\
\mathcal{S}_{11,{4}}(\zeta_{11}, z)&=-(\zeta_{11}^3 + \zeta_{11}^8) \, \frac{\eta(11z)^2}{\eta_{11,{5}}(z)},
\\
\mathcal{S}_{11,{5}}(\zeta_{11}, z)&=\frac{\eta(11z)^2}{\eta_{11,{1}}(z)},
\\
\mathcal{S}_{11,{6}}(\zeta_{11}, z)&=- (2+\zeta_{11}^2+\zeta_{11}^3 +\zeta_{11}^4+ \zeta_{11}^5+\zeta_{11}^6+\zeta_{11}^7+\zeta_{11}^8+\zeta_{11}^9) \, \frac{\eta(11z)^2 \, \eta_{11,{5}}(z)}{\eta_{11,{2}}(z)\, \eta_{11,{3}}(z)},
\\
\mathcal{S}_{11,{7}}(\zeta_{11}, z)&=(\zeta_{11}^2 + \zeta_{11}^9) \,  \frac{\eta(11z)^2 \, \eta_{11,{3}}(z)}{\eta_{11,{1}}(z)\, \eta_{11,{4}}(z)},
\\
\mathcal{S}_{11,{8}}(\zeta_{11}, z)&=(1+\zeta_{11}^3 + \zeta_{11}^8) \,  \frac{\eta(11z)^2 \, \eta_{11,{2}}(z)}{\eta_{11,{1}}(z)\, \eta_{11,{3}}(z)},
\\
\mathcal{S}_{11,{9}}(\zeta_{11}, z)&=(1+\zeta_{11}^2+\zeta_{11}^4 + \zeta_{11}^7+\zeta_{11}^9) \, \frac{\eta(11z)^2}{\eta_{11,{2}}(z)},
\\
\mathcal{S}_{11,{10}}(\zeta_{11}, z)&=- (\zeta_{11}^2 + \zeta_{11}^4+\zeta_{11}^7+\zeta_{11}^9) \,  \frac{\eta(11z)^2 \, \eta_{11,{4}}(z)}{\eta_{11,{2}}(z)\, \eta_{11,{5}}(z)}.
\end{align*}

We consider the $m=7$ element of the dissection 

$$\mathcal{S}_{11,{7}}(\zeta_{11}, z)= (\zeta_{11}^2+\zeta_{11}^9) \, \frac{\eta(11z)^2\eta_{11,{3}}(z)}{\eta_{11,{1}}(z)\eta_{11,{4}}(z)},$$ and find ${\mathcal{S}_{11,7}(\zeta_{11}, z)}|[A]_1$ for $
A = \begin{pmatrix}
      3 & 1 \\ 11 & 4 
    \end{pmatrix} \in \Gamma_0(11).
$

Using the transformations for $\eta(z)$ and the theta function $f_{N,\rho}(z)$ (\cite[Theorem 6.12 \& 6.14, p.243]{Ga19}) and the fact that $f_{N,\rho}(z)=f_{N,N+\rho}(z)=f_{N,-\rho}(z)$, we get
\begin{align*}
{\mathcal{S}_{11,7}(\zeta_{11}, z)}|[A]_1
&={(\zeta_{11}^2 + \zeta_{11}^9) \, \frac{\eta(11z)^3 \, f_{11,3}(z)}{f_{11,1}(z)\, f_{11,4}(z)}}|[A]_1\\
&=(\zeta_{11}^2 + \zeta_{11}^9) \, \frac{\eta(11z)^3 \, f_{11,9}(z)}{f_{11,3}(z)\, f_{11,12}(z)}\frac{e^{\frac{27 \pi i}{11}}}{e^{\frac{3 \pi i}{11}}.e^{\frac{48 \pi i}{11}}}\\
&=-\frac{\sin(\frac{\pi}{11})}{\sin(\frac{4\pi}{11})}\,
\exp\left({\frac{9\pi i}{11}}\right)\,(1+\zeta_{11}^3 + \zeta_{11}^8) \, \frac{\eta(11z)^2 \, \eta_{11,2}(z)}{\eta_{11,3}(z)\, \eta_{11,1}(z)}\\
&=-\frac{\sin(\frac{\pi}{11})}{\sin(\frac{4\pi}{11})}\,\exp\left({\frac{9\pi i}{11}}\right)\,\mathcal{S}_{11,8}(\zeta_{11}^4,z),
\end{align*}
which agrees with our symmetry result  (\thm{symmthm1}). We also make note of the fact that $m=7$ and $8$ are both quadratic non-residues modulo $11$. We can thus conclude that in the light of our symmetry theorem, the identities for $\mathcal{K}_{p,m}(\zeta_{p},z)$ in terms of generalized eta functions can be fully determined if we know one identity each for the cases when $m=0$, $m$ is a quadratic residue mod $p$ and $m$ is a quadratic non-residue mod $p$.

\section{Lower bounds for orders of $\mathcal{S}_{p,m}(\zeta_p,z)$ at cusps}
\label{sec:cuspords}

In this section, we calculate lower bounds for the orders of $\mathcal{S}_{p,m}(\zeta_p,z)$ at the cusps of $\Gamma_1(p)$, which we use in proving the $\mathcal{S}_{p,m}(\zeta_p,z)$ identities in the subsequent section. 
\\\\
We first establish a few other useful transformation results.
\\\\
Let $p>3$ be prime, $1\le\ell\le(p-1)$, and define $\Gell{}{\ell}{p}{z} := \dfrac{\eta(z)^2}{E_{\ell,0}(z)}$, where $E_{\ell,0}(z)=q^{B(\ell/p)/2}\prod\limits_{m=1}^{\infty}\left(1-q^{m-1+\frac{\ell}{p}}\right)\left(1-q^{m-\frac{\ell}{p}}\right)$.
\\
We list the transformations of $E_{0,\ell}(z)$ under matrices $S=\SMat$ and $T=\TMat$, which are the generators of $\SL_2(\mathbb{Z})$, and that of $\Gell{}{\ell}{p}{z}$ under $S$.

\begin{theorem} 
\label{thm:F2trans}
\begin{enumerate}
\item $\stroke{E_{0,\ell}(z)}{S}{1}=-i \zeta_{2p}^{\ell}E_{\ell,0}(z)$.
\item $\stroke{\Gell{}{\ell}{p}{z}}{S}{1}=\zeta_{2p}^{-\ell}\Fell{}{-\ell}{p}{z}$.
\item $\stroke{E_{0,\ell}(z)}{T}{1}=e^{\pi i/6}E_{0,\ell}(z)$.
\end{enumerate}
\end{theorem}

\begin{proof}
Proofs of (1) and (3) follow easily using the definition of $E_{g,h}(z)$. (2) follows from (1) and the transformation of $\eta(z)$ under $S$ i.e. $\stroke{\eta(z)}{S}{1}=\exp(\frac{-\pi i}{4})\,\eta(z)$ ({\cite[Theorem 2, p.51]{Kn-book}}). 
\end{proof}
Let
$
\mathfrak{V}_p:= \left\{\Fell{}{\ell}{p}{z},\,
             \Gell{}{\ell}{p}{z}\,:\, 0< \ell < p\right\}.
$
\\\\
The action of $\SL_2(\mathbb{Z})$ on each element can be given explicitly using Theorems {\cite[Theorem 2, p.51]{Kn-book}} and \thm{F2trans}.  Also each $\mathcal{F}\in\mathfrak{V}_p$ has a $q$-expansion
$$\mathcal{F}(z) = \sum_{m\ge m_0} a(m) \exp\Lpar{2\pi i z \frac{m}{24p^2}},$$ 
where $a(m_0)\ne 0$. We define $$\ord(\mathcal{G};\infty) := \frac{m_0}{24p^2}.$$
\\
For any cusp $\frac{a}{c}$ with $(a,c)=1$ we define 
$$\ord\Lpar{\mathcal{F};\frac{a}{c}} := \ord(\stroke{\mathcal{F}}{A}{1};\infty),$$
where $A\in\SL_2(\mathbb{Z})$ and $A\infty = \frac{a}{c}$.  We note that $\stroke{\mathcal{F}}{A}{1}\in \mathfrak{V}_p$. We also note that this definition coincides with the definition of invariant order at a cusp \cite[p.2319]{Br-Sw}, \cite[p.275]{Bi89}.
The order of each function $\mathcal{F}(z)\in \mathfrak{V}_p$ is defined in the natural way i.e. $$\ord\Lpar{\mathcal{F};\frac{a}{c}} := 2\,\ord\Lpar{\eta(z);\frac{a}{c}}\,+\, \ord\Lpar{1/E_{*,*};\frac{a}{c}},$$ where $\ord\Lpar{\eta(z);\frac{a}{c}}$ is the usual invariant order of $\eta(z)$ at the cusp $\frac{a}{c}$ \cite[p.34]{Ko-book}.
\\\\
We determine $\ord\Lpar{\mathcal{F};\infty}$ for each $\mathcal{F}(z) \in \mathfrak{V}_p$.  After some calculation we find
\begin{prop}
\label{propo:Fords}
Let $p>3$ be prime. Then
$$\ord\Lpar{\Fell{}{\ell}{p}{z};\infty} = 0,$$
$$\ord\Lpar{\Gell{}{\ell}{p}{z};\infty} = \frac{\ell}{2p}-\frac{\ell^2}{2p^2}.$$
\end{prop}
We also need \cite[Corollary 2.2]{Ko-book}.
\begin{prop}
\label{propo:etaord}
Let $N\ge 1$ and let
$$
F(z) = \prod_{m\mid N} \eta(m z)^{r_m},
$$
where each $r_m\in\mathbb{Z}$. Then for $(a,c)=1$,
$$
\ord\Lpar{F(z);\frac{a}{c}} = \sum_{m\mid N} \frac{(m,c)^2 r_m}{24m}.
$$
\end{prop}

From \cite[Corollary 4, p.930]{Ch-Ko-Pa} and \cite[Lemma 3, p.929]{Ch-Ko-Pa}
we have
\begin{prop}
\label{propo:cusps1}
Let $p>3$ be prime. Then we have the following set of inequivalent cusps for $\Gamma_1(p)$ and their corresponding fan widths.
$$
\begin{matrix}
\mbox{Cusp:}\quad  & i\infty,\, & 0,\,  & \frac{1}{2},\, & \frac{1}{3},\, & \dots,\, & \frac{1}{\tfrac{1}{2}(p-1)},\, & \frac{2}{p},\, & \frac{3}{p},\, & \dots,\, & \frac{\tfrac{1}{2}(p-1)}{p},\\
\mbox{Fan width:}\quad & 1,\, &  p,\, & p,\, & p,\, & \dots,\, & p,\, & 1,\, & 1,\, & \dots,\, & 1.     
\end{matrix}
$$
\end{prop}

We next calculate lower bounds of the invariant order of $\mathcal{S}_{p,m}(\zeta_p,z)$ at each cusp of $\Gamma_1(p)$.

\begin{theorem}
\label{thm:CKpmords}
Let $p>3 $ be prime, and suppose $0 \le m \le p-1$. Then
\begin{enumerate}
\item[(i)]
$$
\ord\Lpar{\mathcal{S}_{p,m}(\zeta_p,z);0}
\quad
\begin{cases}
\ge 0 & \mbox{if $p=11$},\\
=-\frac{1}{24p}(p-1)(p-11) & \mbox{otherwise};
\end{cases}
$$
\item[(ii)]
$$
\ord\Lpar{\mathcal{S}_{p,m}(\zeta_p,z);\frac{1}{n}}
\quad
\begin{cases}
= - \dfrac{1}{2p}\left[\dfrac{(p-1)(p+1)}{12}+n (n-p)\right] &
\mbox{if $2\le n <\tfrac{p}{2}-\tfrac{1}{2}\sqrt{\frac{2p^2+1}{52}}$},\\
\ge 0 & \mbox{otherwise};
\end{cases}
$$
and
\item[(iii)]
$$
\ord\Lpar{\mathcal{S}_{p,m}(\zeta_p,z);\frac{n}{p}}
\ge \Lpar{\frac{p^2-1}{24p}}.
$$
\end{enumerate}
\end{theorem}
\begin{proof}
We derive lower bounds for $\ord\Lpar{\mathcal{S}_{p,m}(\zeta_p,z);\zeta}$ for each cusp $\zeta$ of $\Gamma_1(p)$ not equivalent to $i\infty$.
\\\\
We have
$$\mathcal{S}_{p,m}(\zeta_p,z)=q^{\frac{1}{24}}(1-\zeta_p)\,\stroke{\SFell{}{1}{p}{z}}{U_{p,m}}{1}=\frac{1}{\sqrt{p}}(1-\zeta_p)\sum\limits_{k=0}^{p-1}\zeta_p^{-km}\stroke{\SFell{}{1}{p}{z}}{T_k}{1}$$
where
\beq
\SFell{}{1}{p}{z} = \frac{\eta(p^2 z)}{\eta(z)}\, \Fell{}{1}{p}{z}.
\label{eq:SFelldef}
\eeq

We calculate 
$$
\stroke{\SFell{}{1}{p}{z}}{T_k\,A}{1},
$$
for each $0 \le k \le p-1$ and each $A=\AMat\in\SLZ$.

\subsubsection*{Case 1} 
\label{subsubsec:CuspConcase1}
$a+kc\not\equiv0\pmod{p}$. 
Choose $0\le k'\le p-1$ such that
$$
(a+kc)\, k' \equiv (b+kd) \pmod{p}.
$$
Then
$$
T_k \, \gA = C_k \, T_{k'},
$$
where
$$
C_k = T_k \, \gA \, T_{k'}^{-1} = 
\begin{pmatrix} a+ck & \tfrac{1}{p}(-k'(a+kc) + b+kd) \\
pc & d -k'c
\end{pmatrix} \in \Gamma_0(p).
$$
From Proposition \propo{F1trans1} we have
\begin{align}
\stroke{\SFell{}{1}{p}{z}}{T_k\,A}{1}
=
\stroke{\SFell{}{1}{p}{z}}{C_k\,T_{k'}}{1}
\label{eq:Tk1}\\
=
\Lpar{\stroke{F_p(z)}{C_k\,T_{k'}}{0}}\,
\Lpar{\mu(C_k,1)\,\stroke{\Fell{}{\overline{d-k'c}}{p}{z}}{T_{k'}}{1}},
\nonumber
\end{align}
where $\overline{x}$ implies that $x\in \mathbb{Z}$ is reduced modulo $p$, and
\beq
F_p(z) = \frac{\eta(p^2 z)}{\eta(z)}.
\eeq
\subsubsection*{Case 2} 
\label{subsubsec:CuspConcase2}
$a+kc\equiv0\pmod{p}$. In this case we find that
$$
T_k \, \gA = D_k \, P,
$$
where
$$
P = \begin{pmatrix} p & 0 \\ 0 & 1\end{pmatrix},
$$
$$
D_k = \begin{pmatrix} \frac{1}{p}(a+kc) & b + kd \\ c & pd \end{pmatrix}\in
\SL_2(\mathbb{Z}),
$$
and  
$$
E_k = D_k\,S = 
\begin{pmatrix}
b+kd  &  \frac{-1}{p}(a+kc)\\
pd & -c
\end{pmatrix}\in\Gamma_0(p). 
$$
From Proposition \propo{F1trans1} we have
$$
\stroke{\Fell{}{1}{p}{z}}{E_k}{1}
= \mu(E_k,1) \, \Fell{}{\overline{-c}}{p}{z}.
$$
By Theorem \thm{F2trans}(2) we have
$$
\stroke{\Fell{}{1}{p}{z}}{D_k}{1}
= \mu(E_k,1) \, \stroke{\Fell{}{\overline{-c}}{p}{z}}{S^{-1}}{1}
= \mu(E_k,1) \, \zeta_{2p} \, \Gell{}{\overline{c}}{p}{z},
$$
and
$$
\stroke{\Fell{}{1}{p}{z}}{T_k\,A}{1}
= \mu(E_k,1) \, \zeta_{2p} \, \stroke{\Gell{}{\overline{c}}{p}{z}}{P}{1},
$$
so that

\begin{equation}
\stroke{\SFell{}{1}{p}{z}}{T_k\,A}{1}
=
\Lpar{\stroke{F_p(z)}{D_k\,P}{0}}\,
\Lpar{\mu(E_k,1) \, \zeta_{2p} \, \stroke{\Gell{}{\overline{c}}{p}{z}}{P}{1}}.
\label{eq:Tk2}
\end{equation}

Now we are ready to exam each cusp $\zeta$ of $\Gamma_1(p)$.
We choose 
$$
\gA = \AMat \in \SL_2(\mathbb{Z}),\quad\mbox{so that $\gA(\infty) = \frac{a}{c}=\zeta$.}
$$

\subsubsection*{(i) $\zeta=0$}
Here $a=0$, $c=1$ and we assume $0 \le k \le p-1$.
If $k\ne0$ then applying \eqn{Tk1} we have
\begin{align*}
\ord\Lpar{\SFTk;0} = 
&= \frac{1}{p}\, \ord\Lpar{F_p(z);\frac{k}{p}} + 
  \frac{1}{p} \, \ord \Lpar{\Fell{}{\overline{d-k'c}}{p}{z}; i\infty}\\
&= 0 + 0 = 0,
\end{align*}
by Propositions \propo{etaord} and \propo{Fords}. Now applying \eqn{Tk2}
with $k=0$ we have
\begin{align*}
\ord\Lpar{\SFTZ;0}
&= p\, \ord\Lpar{F_p(z);0} + 
   p\, \ord \Lpar{\Gell{}{\overline{c}}{p}{z}; i\infty}\\
&= -\frac{1}{24p}(p^2-1) + p\left(\frac{1}{2p}-\frac{1}{2p^2}\right)
\\
&= -\frac{1}{24p}(p-1)(p-11)
\end{align*}
again by Propositions \propo{etaord} and \propo{Fords}. The result (i) follows since
$$
\ord\Lpar{\mathcal{S}_{p,m}(z);0} \ge \, \min_{0 \le k \le p-1} 
\ord\Lpar{\SFTk;0}.
$$
\subsubsection*{(ii) $\zeta=\frac{1}{n}$, where $2 \le n \le \tfrac{1}{2}(p-1)$}
Let $A=\begin{pmatrix}1 & 0 \\ n & 1 \end{pmatrix}$ so that $A(\infty)=1/n$.
If $kn\not\equiv-1\pmod{p}$ we apply \eqn{Tk1} with 
$C_k = \begin{pmatrix}1 + kn & * \\ pn & 1-k'n\end{pmatrix}$, and find that
\begin{align*}
\ord\Lpar{\SFTk;\frac{1}{n}}
&= \frac{1}{p}\, \ord\Lpar{F_p(z);\frac{1+kn}{pn}} + 
  \frac{1}{p} \, \ord \Lpar{\Fell{}{\overline{1-kn}}{p}{z}; i\infty}\\
&= 0 + 0 = 0.
\end{align*}
Now we assume $kn\equiv-1\pmod{p}$ and we will apply \eqn{Tk2}. We have
\begin{align*}
\ord\Lpar{\SFTk;\frac{1}{n}} &= p\, \ord\Lpar{F_p(z);\frac{(1+kn)/p}{n}} + 
   p\, \ord \Lpar{\Gell{}{n}{p}{z}; i\infty}
\\
&= -\frac{1}{24p}(p^2-1) + p\left(\frac{n\ell}{2p}-\frac{n^2\ell^2}{2p^2}\right) 
\\
&                              
\begin{cases}
= - \dfrac{1}{2p}\left[\dfrac{(p-1)(p+1)}{12}-n (p-n)\right]
&\mbox{if $2\le n <\tfrac{p}{2}-\tfrac{1}{2}\sqrt{\frac{2p^2+1}{52}}$},\\
\geq 0 & \mbox{otherwise},
\end{cases}
\end{align*}
and the result (ii) follows.

\subsubsection*{(iii) $\zeta=\frac{n}{p}$, where $2 \le n \le \tfrac{1}{2}(p-1)$}
Choose $b$ and $d$ so that $A=\begin{pmatrix}n & b \\ p & d \end{pmatrix}
\in \SL_2(\mathbb{Z})$ and 
$A(\infty)=n/p$. Since $n\not\equiv0\pmod{p}$ we may apply \eqn{Tk1} for each $k$. We find that
$C_k = \begin{pmatrix}n + kp & * \\ p^2 & d-k'p\end{pmatrix}$, and 
\begin{align*}
\ord\Lpar{\SFTk;\frac{n}{p}}
&= \frac{1}{p}\, \ord\Lpar{F_p(z);\frac{n+kp}{p^2}} + 
  \frac{1}{p} \, \ord \Lpar{\Fell{}{\overline{d-k'p}}{p}{z}; i\infty}\\
&= \frac{p^2-1}{24p} + 0 = \frac{p^2-1}{24p}.
\end{align*}
The result (iii) follows.
\end{proof}

\section{Crank theorem and crank mod $p$ identities}
\label{sec:crankids}

\subsection{Dyson's rank conjectures and the crank theorem}
10 years after they were proposed by Dyson \cite{Dy44}, Atkin and Swinnerton Dyer \cite{At-Sw} came up with a proof of the Dyson's rank conjectures stated below. 
\begin{drconj}[1944]
\label{conj:DRC}
For all nonnegative integers $n$,
\begin{align}
N(0,5,5n+4) &= N(1,5,5n+4) = \cdots  = N(4,5,5n+4) = \tfrac{1}{5}p(5n + 4),
\label{eq:Dysonconj5}\\
N(0,7,7n+5) &= N(1,7,7n+5) = \cdots  = N(6,7,7n+5) = \tfrac{1}{7}p(7n + 5).    
\label{eq:Dysonconj7}\\\nonumber
\end{align}
\end{drconj}

Garvan, in \cite{Ga19} notes that the proof in  \cite{At-Sw} involved finding and  proving identities for basic hypergeometric functions, theta-functions and Lerch-type series using the theory of elliptic functions. It also involved identifying the generating functions for rank differences $N(0,p,pn+r)-N(k,p,pn+r)$  for $p=5$, $7$ for each $ 1 \le k \le\tfrac{1}{2}(p-1)$ and each $r=0$, $1$,\dots $p-1$. Subsequently in his paper \cite[Section 6.3]{Ga19}, Garvan gives a new proof of these conjectures using the theory of modular forms, by calculating the orders of the rank dissection element $\mathcal{K}_{p,0}^{}(\zeta_{p},z)$ at the cusps of $\Gamma_1(p)$ for $p=5$ and $7$ and making use of the Valence formula. In this section, we employ Garvan's idea in \cite{Ga19} to give a new proof of the crank theorem modulo $11$ (Equation \eqn{crank11}) using the Valence formula and the orders of $\mathcal{S}_{11,0}^{}(\zeta_{11},z)$ at the cusps of $\Gamma_1(11)$ calculated in the previous section. It is easy to observe that in our setup, the crank theorem is equivalent to showing
\beq
\mathcal{S}_{11,0}^{}(\zeta_{11},z) = 0.
\label{eq:CK11}
\eeq

\begin{theorem}[The Valence Formula \cite{Ra}(p.98)]
\label{thm:val}
Let $f\ne0$ be a modular form of weight $k$ with respect to a subgroup $\Gamma$ of finite index
in $\Gamma(1)=\SL_2(\mathbb{Z})$. Then
\beq
\ORD(f,\Gamma) = \frac{1}{12} \mu \, k,
\label{eq:valform}
\eeq
where $\mu$ is index $\widehat{\Gamma}$ in $\widehat{\Gamma(1)}$ for $\widehat{\Gamma}=\Gamma/\{\pm I\}$,
$$
\ORD(f,\Gamma) := \sum_{\zeta\in R^{*}} \ORD(f,\zeta,\Gamma),
$$
$R^{*}$ is a fundamental region for $\Gamma$,
and
\begin{equation}
\label{eq:ORDdef}
\ORD(f;\zeta;\Gamma) = n(\Gamma;\zeta)\, \ord(f;\zeta),
\end{equation}
for a cusp $\zeta$ and
$n(\Gamma;\zeta)$ denotes the fan width of the cusp $\zeta \pmod{\Gamma}$.
\end{theorem}
\begin{remark}
For $\zeta\in\mathcal{H}$,
$\ORD(f;\zeta;\Gamma)$ is defined in terms of the invariant order $\ord(f;\zeta)$, which  is interpreted in the usual sense. See \cite[p.91]{Ra} for details of this and the notation used.
\end{remark}

In the following table, we give the order of $\mathcal{S}_{11,0}^{}(\zeta_{11},z)$ at the cusps using Theorem \thm{CKpmords}. Here $2\leq m \leq \frac{p-1}{2}$.

\begin{center}
\begin{tabular}{c c c c}
cusp $\zeta$ & $n(\Gamma_1(11);\zeta)$ & $\ord(\mathcal{S}_{11,0}^{}(\zeta_{11},z);\zeta)$ 
& $\ORD(\mathcal{S}_{11,0}^{}(\zeta_{11},z),\Gamma_1(11),\zeta)$ \\
$i\infty$      & $1$  & $\ge 1$             & $\ge 1$ \\
$0$            & $11$ & $\ge 0$             & $\ge 0$ \\
$\frac{1}{m}$  & $11$ & $> 0$               & $> 0$ \\
$\frac{m}{11}$ & $1$  & $\ge \frac{5}{11}$  & $\ge 1$ \\\\
\end{tabular}
\end{center}

An easy calculation by hand shows that $M(0,11,6)=\cdots = M(10,11,6)$. Expanding $\mathcal{S}_{11,0}^{}(\zeta_{11},z)$ we see that this implies the coefficient of $q^1$ is $0$ and thus we have that 
\\
$\ORD(\mathcal{S}_{11,0}^{}(\zeta_{11},z),\Gamma_1(11),i \infty) \geq 2$. Hence $\ORD(\mathcal{S}_{11,0}^{}(\zeta_{11},z);\Gamma_1(11)) \geq 6.$ But $\mu k = \frac{60}{12} = 5$. The Valence Formula implies that $\mathcal{S}_{11,0}^{}(\zeta_{11},z)$ is identically zero which proves the crank theorem.

\begin{remark}
The same method could be applied to prove the mod $5$ and $7$ and cases of the crank theorem i.e. Equations \eqn{crank5} and \eqn{crank7}.
\end{remark} 
\subsection{Crank mod $p$ identities}

In this section, we give identities for  $\mathcal{S}_{p,m}(\zeta_{p},z)$ when $p=13, 17$ and  $19$ in terms of generalized eta-functions given in Definition \refdef{geneta}. Crank mod $13$ identities were first studied by Bilgici in \cite{Bi13} similar to the work on rank mod $11$ identities by Atkin and Hussain \cite{At-Hu} and rank mod $13$ identities considered by O'Brien in his thesis \cite{OB-thesis}. In \cite{Ga-Ri}, we also gave such identities for elements of the the rank generating function. In general, these identities are of the form 
\beq
\mathcal{S}_{p,m}(\zeta_{p},z)= \sum\limits_{k=1}^r c_{p,m,k}\, j_{p,m,k}(z)
\label{eq:Kpmid2}
\eeq
where $j_{p,m,k}(z)$ are eta-quotients comprising of generalized eta-functions, and $c_{p,m,k}$ are cyclotomic coefficients. From Theorem \thm{Kpmthm}, we know that $\mathcal{S}_{p,m}(\zeta_{p},z)$ is a weakly holomorphic modular form of weight $1$ on $\Gamma(p)$.  The proof of the identities primarily involve establishing the equality using the Valence formula and showing that the RHS is also a weakly holomorphic modular form of weight $1$ on $\Gamma(p)$.  
\\\\
In \cite{Ga-Ri}, we derived conditions for an eta-quotient $j(p,\overrightarrow{n},z)$ to be a weakly holomorphic modular form of weight $1$ on $\Gamma(p)$. 
\begin{theorem}[\cite{Ga-Ri}, Theorem 3.1]\label{thm:modular}
Let $p, \overrightarrow{n}$ and $j(p,\overrightarrow{n},z)$ be as in Definition \refdef{geneta}.
Then $j(p,\overrightarrow{n},z)$ is a weakly holomorphic modular form of weight 1 on $\Gamma(p)$ satisfying the modularity condition $j(p,\overrightarrow{n},z)\stroke{}{A}{1}=\exp(\tfrac{2\pi i b m}{p})j(p,\overrightarrow{n},z)$ for $A=\begin{pmatrix} a & b \\ c & d\end{pmatrix} \in \Gamma_1(p)$ provided the following conditions are met :
\\
\begin{align*}
&(1)~n_0+\sum_{k=1}^{\frac{p-1}{2}}n_k=2,
\\
&(2)~n_0+3\sum_{k=1}^{\frac{p-1}{2}}n_k \equiv 0 \pmod{24}, 
\\
&(3)~\sum_{k=1}^{\frac{p-1}{2}}k^2 n_k \equiv 2m \pmod{p}.
\end{align*}
\end{theorem}

In Section 7 of \cite{Ga-Ri}, we developed an algorithm for proving rank mod $p$ identities that utilizes the Valence Formula. Here, we follow the steps of the algorithm to prove crank identities modulo $13$. 
\\\\
For the cases $p=17$ and $19$, we simply give the form of the identities omitting the values of the coefficients, for each of the cases $m=0$ and when $m$ is a quadratic residue and a non-residue modulo $p$. Their proofs follow the same technique using the algorithm as it does for $p=13$.
\\\\
To help calculate the order of the generalized eta functions at the cusps, we also need the following result from \cite[Prop.7.4]{Ga-Ri}.

\begin{prop}
\label{propo:ncuspORDS}
Let $p>3$ be prime and suppose $s=\frac{a}{c}$ is one of the cusps listed
in Proposition \propo{cusps1} with $i\infty$ represented by $\frac{1}{p}$.
Then 
\begin{enumerate}
\item[(i)]
If $(c,p)=1$ then
$$
\ORD\left(j(p,\overrightarrow{n},z),s,\Gamma_1(p)\right)
= \frac{1}{24}\left( n_0 + 3 \sum_{j=1}^{(p-1)/2} n_j\right).
$$
\item[(ii)] 
If $c=p$ then
$$
\ORD\left(j(p,\overrightarrow{n},z),s,\Gamma_1(p)\right)
= \frac{p}{24}\left( n_0 + 12 \sum_{j=1}^{(p-1)/2} 
n_j \left( \frac{aj}{p} - \FL{\frac{aj}{p}} - \frac{1}{2}\right)^2\right).
$$
\end{enumerate}
\end{prop}

\subsection{Crank mod 13 identities}

\subsubsection{\textbf{Identity for $\mathcal{S}_{13,0}^{}$}}

The following is an identity for $\mathcal{S}_{13,0}^{}(\zeta_{13},z)$ in terms of generalized eta-functions :
\\
\begin{align}
\mathcal{S}_{13,0}^{}(\zeta_{13},z)&=(q^{13};q^{13})_\infty \sum_{n=1}^\infty \Lpar{\sum_{k=0}^{12} M(k,13,13n-7)\,\zeta_{13}^k}q^n \label{eq:crank13id1} \\
\nonumber
\\
&=\sum_{r=1}^{6} c_{13,r}\,j(13,\pi_r(\overrightarrow{n_1}),z),
\nonumber
\end{align}
\\
where 
\begin{align*}
\overrightarrow{n_1}=(15,-1,-3,-2,-2,-2,-3),
\end{align*}\\
and the coefficients are :
\begin{align*}
c_{13,6} &= 2\, \zeta_{13}^{11}+\zeta_{13}^{9}+\zeta_{13}^{8}+\zeta_{13}^{7}+\zeta_{13}^{6}+\zeta_{13}^{5}+\zeta_{13}^{4}+2\,\zeta_{13}^{2}+2,
\\
c_{13,1} &= -(\zeta_{13}^{11} + \zeta_{13}^{10} + \zeta_{13}^7 + \zeta_{13}^6 + \zeta_{13}^3 + \zeta_{13}^2),
\\
c_{13,5} &= \zeta_{13}^{11}+\zeta_{13}^{10}+\zeta_{13}^{9}+\zeta_{13}^{8}+\zeta_{13}^{5}+\zeta_{13}^{4}+\zeta_{13}^{3}+\zeta_{13}^{2}+1,
\\
c_{13,2} &= \zeta_{13}^{10}-\zeta_{13}^{9}+ \zeta_{13}^{8}+ \zeta_{13}^{5}-\zeta_{13}^{4}+\zeta_{13}^{3}+1,
\\
c_{13,4} &=  -\left(\zeta_{13}^{10}+\zeta_{13}^{9}+\zeta_{13}^{7}+\zeta_{13}^{6}+\zeta_{13}^{4}+\zeta_{13}^{3}+2\right),
\\
c_{13,3} &= -\left(\zeta_{13}^{8}+\zeta_{13}^{5}\right).
\\
\end{align*}
We follow the steps of the aforementioned algorithm in the process of proving the identity.
\\
\begin{enumerate}
\item[Step 1] We check the conditions for modularity as in Theorem \thm{modular} for $j(13,\pi_r(\overrightarrow{n_1}),z), 1\leq r \leq 6$ involved in \eqn{crank13id1}. Here, $p=13, m=0, n_0=15$. In \cite{Ga-Ri}, Theorem 3.6 tells us that it suffices to check modularity for $r=1$. With $\overrightarrow{n_1}=(15,-1,-3,-2,-2,-2,-3)$, we easily see that 
\begin{align*}
n_0+\displaystyle \sum_{k=1}^{5}n_k&=2, 
\\
n_0+3 \displaystyle \sum_{k=1}^{5}n_k &= -24 \equiv 0\pmod {24},
\\
\displaystyle \sum_{k=1}^{5}k^2 n_k &= -221 \equiv 0 \pmod {13},
\end{align*}
as required.
\\\\
Since $m=0$, we skip Step 2 in accordance with our algorithm.
\\
\item[Step 3] Using Proposition \propo{ncuspORDS}, we calculate the orders of each of the six functions $f$ at each cusp of $\Gamma_1(13)$.\\

\begin{center}
$\ORD(f,\zeta,\Gamma_1(13))$\\
\begin{tabular}{c c c c c c c c c c}\\
&&&& cusp $\zeta$  \\
$f$ & $i\infty$ & $0$ & $1/n$ & $2/13$ & $3/13$ & $4/13$ & $5/13$ & $6/13$  \\
$j(13,\pi_1(\overrightarrow{n_1}),z)$ &  $3$  & $-1$ & $-1$  & $2$    & $3$    &  $2$   &  $2$   &  $1$   \\
$j(13,\pi_2(\overrightarrow{n_1}),z)$ &  $2$  & $-1$ & $-1$  & $2$    & $1$    &  $2$   &  $3$   &  $3$   \\
$j(13,\pi_3(\overrightarrow{n_1}),z)$ &  $3$  & $-1$ & $-1$  & $1$    & $2$    &  $3$   &  $2$   &  $2$   \\
$j(13,\pi_4(\overrightarrow{n_1}),z)$ &  $2$  & $-1$ & $-1$  & $2$    & $3$    &  $3$   &  $1$   &  $2$   \\
$j(13,\pi_5(\overrightarrow{n_1}),z)$ &  $2$  & $-1$ & $-1$  & $3$    & $2$    &  $1$   &  $3$   &  $2$   \\
$j(13,\pi_6(\overrightarrow{n_1}),z)$ &  $1$  & $-1$ & $-1$  & $3$    & $2$    &  $2$   &  $2$   &  $3$   
\end{tabular}
\end{center}
where $2\le n\le 6$.
\\
\item[Step 4] Considering the LHS of equation \eqn{crank13id1}, we calculate lower bounds $\lambda(13,0,\zeta)$ of orders $ORD\left(\mathcal{S}_{13,0}(\zeta_{13},z),\zeta,\Gamma_1(13)\right)$ at cusps $\zeta$ of of $\Gamma_1(13)$.
\\
\begin{center}
\begin{tabular}{c c c c c}
cusp $\zeta$ &   $\lambda(13,0,\zeta)$\\
$i\infty$    &           \\
$0$          &   $-1$    \\
$1/n$        &   $0$     \\
$n/13$       &   $\CL{7/13}=1$  \\
\end{tabular} 
\end{center} 
where $2\le n\le 6$ using Theorem \thm{CKpmords}. We note that each value is an integer.
\\
\item[Step 5] We summarize the calculations in Steps 4 and 5 of the algorithm (see \cite{Ga-Ri}, Section 7) in a Table.
The gives lower bounds for the LHS and RHS of equation \eqn{crank13id1},
at the cusps $\zeta$.
\begin{center}
\begin{tabular}{c c c c c}
cusp $\zeta$ & $n(\Gamma_1(13);\zeta)$ & $\ORD(LHS;\zeta)$ & $\ORD(RHS;\zeta)$ 
& $\ORD(LHS-RHS;\zeta)$ \\
$i\infty$    & $1$  &            &            &             \\
$0$          & $13$ &  $\ge-1$   &  $\ge-1$   &   $\ge -1$  \\
$1/n$        & $13$ &  $\ge0$    &  $\ge-1$   &   $\ge -1$  \\
$n/13$       & $1$  &  $\ge1$    &  $1$       &   $\ge1$    \\\\
\end{tabular} 
\end{center} 
where $2\le n\le 6$.
\\\\
The constant $B$ is the sum of the lower bounds in the last column, so that $B=-1$.
\\
\item[Step 6] The LHS and RHS are weakly holomorphic modular forms of weight $1$ on $\Gamma_1(13)$. So in the Valence Formula, $\frac{\mu k}{12} = 7$. The result follows provided we can show that $\ORD(LHS-RHS,i\infty,\Gamma_1(13))\geq 9$. This is easily verified using MAPLE (see \href{https://github.com/rishabh-sarma/crank-identities-modulo-13-17-19}{\textit{https://github.com/rishabh-sarma/crank-identities-modulo-13-17-19}}).
\end{enumerate}

\subsubsection{\textbf{A quadratic residue case}}

The following is an identity for $\mathcal{S}_{13,2}^{}(\zeta_{13},z)$ in terms of generalized eta-functions : 

\begin{align}
\mathcal{S}_{13,2}^{}(\zeta_{13},z)&=q^\frac{2}{13}(q^{13};q^{13})_\infty \Bigg(\sum_{n=1}^\infty \Lpar{\sum_{k=0}^{12} M(k,13,13n-5)\,\zeta_{13}^k}q^n\Bigg) \label{eq:crank13id2}\\
\nonumber
\\
&=\frac{f_{13,1}(z)}{f_{13,6}(z)}\sum_{k=-1}^{0}\sum_{r=1}^{6} \left(\dfrac{\eta(13z)}{\eta(z)}\right)^{2k}c_{13,r,k}\,j(13,\pi_r(\overrightarrow{n_1}),z),
\nonumber
\end{align}
where 
\begin{align*}
\overrightarrow{n_1}=(15,-1,-3,-2,-2,-2,-3),
\end{align*}
and the coefficients are :

\begin{align*}
c_{13,1,0} &= -(\zeta_{13}^{11}+\zeta_{13}^{9}+\zeta_{13}^{7}+\zeta_{13}^{6}+\zeta_{13}^{4}+\zeta_{13}^{2}+1),
\\
c_{13,2,0} &= 13\, (\zeta_{13}^{11}-\zeta_{13}^{9}-\zeta_{13}^{8}-\zeta_{13}^{5}-\zeta_{13}^{4}+\zeta_{13}^{2}-1),
\\
c_{13,3,0} &= -13\, (\zeta_{13}^{9}+\zeta_{13}^{4}),
\\
c_{13,4,0} &= \zeta_{13}^{11}+\zeta_{13}^{10}+\zeta_{13}^{8}+\zeta_{13}^{5}+\zeta_{13}^{3}+\zeta_{13}^{2}+2,
\\
c_{13,5,0} &= -(\zeta_{13}^{10}+\zeta_{13}^{7}+\zeta_{13}^{6}+\zeta_{13}^{3}),
\\
c_{13,6,0} &= -(\zeta_{13}^{11}+\zeta_{13}^{10}+\zeta_{13}^{9}+2\,\zeta_{13}^{8}+2\,\zeta_{13}^{7}+2\,\zeta_{13}^{6}+2\,\zeta_{13}^{5}+\zeta_{13}^{4}+\zeta_{13}^{3}+\zeta_{13}^{2}),
\\
~~\hspace{3mm}c_{13,1,-1} &= c_{13,2,-1} = c_{13,3,-1}= c_{13,4,-1}= c_{13,5,-1} = 0,
\\
c_{13,6,-1} &= -(\zeta_{13}^{10}+\zeta_{13}^{7}+\zeta_{13}^{6}+\zeta_{13}^{3}).\\
\end{align*}
We follow the steps of the algorithm in the process of proving the above identity.
\\
\begin{enumerate}
\item[Step 1] We check the conditions for modularity as in Theorem \thm{modular} for $\frac{f_{13,1}(z)}{f_{13,6}(z)}j(13,\pi_r(\overrightarrow{n_1}),z)$ and 
\\
$\frac{f_{13,1}(z)}{f_{13,6}(z)}\left(\frac{\eta(13z)}{\eta(z)}\right)^{-2}j(13,\pi_r(\overrightarrow{n_1}),z)$ involved in \eqn{crank13id2}. We first make note of two aspects of our functions. First, from Apostol's \cite[Theorem 4.9, p.87]{Ap76}, we observe that the eta quotient $\left(\frac{\eta(13z)}{\eta(z)}\right)^{2}$ is modular under the congruence subgroup $\Gamma_0(13)$. Also, writing the theta quotient  $\frac{f_{13,1}(z)}{f_{13,6}(z)}$ in its equivalent vector notation $\overrightarrow{n}=(n_0',n_1',n_2',\cdots ,n_6') =(0,1,0,0,0,0,-1)$ in accordance with our Definition \refdef{geneta}, we observe that it doesn't contribute to the weight and 
\begin{align*}
n_0'+3 \displaystyle \sum_{k=1}^{6}n_k' = 0\equiv 0 \pmod{24},
\\
\displaystyle \sum_{k=1}^{6}k^2 n_k'=-35\equiv 4 \pmod{13}. 
\end{align*}
Thus it suffices to check the first step for the functions $j(13,\pi_r(\overrightarrow{n_1}),z), 1 \leq r \leq 6$, which was achieved in Step 1 of the proof for the previous identity $\mathcal{S}_{13,0}^{}(\zeta_{13},z)$.
\\
\item[Step 2] Next, we calculate the orders of the generalized eta-functions at $i\infty$ and considering the identity with zero coefficients removed, we find that $k_0=1$. Thus we divide each generalized eta-function by $j_0=\frac{f_{13,1}(z)}{f_{13,6}(z)}\left(\frac{\eta(13z)}{\eta(z)}\right)^{-2}j(13,\pi_6(\overrightarrow{n_1}),z)$, which has the lowest order at $i\infty$. 
\\
\item[Step 3] Using Proposition \propo{ncuspORDS}, we calculate the orders of each of the six functions $f$ at each cusp of $\Gamma_1(13)$.\\

\begin{center}
$\ORD(f,\zeta,\Gamma_1(13))$\\
\begin{tabular}{c c c c c c c c c c c}\\
&&&& cusp $\zeta$\\
$f$ & $i\infty$ & $0$ & $1/n$ & $2/13$ & $3/13$ & $4/13$ & $5/13$ & $6/13$  \\
$\frac{f_{13,1}(z)}{f_{13,6}(z)}j(13,\pi_1(\overrightarrow{n_1}),z)/j_0$ &  $3$  & $-1$ & $-1$   & $0$    & $2$    &  $1$   &  $1$   &  $-1$   
\\
$\frac{f_{13,1}(z)}{f_{13,6}(z)}j(13,\pi_2(\overrightarrow{n_1}),z)/j_0$ &  $2$  & $-1$ & $-1$   & $0$    & $0$    &  $1$   &  $2$   &  $1$   
\\
$\frac{f_{13,1}(z)}{f_{13,6}(z)}j(13,\pi_3(\overrightarrow{n_1}),z)/j_0$ &  $3$  & $-1$ & $-1$   & $-1$   & $1$    &  $2$   &  $1$   &  $0$   
\\
$\frac{f_{13,1}(z)}{f_{13,6}(z)}j(13,\pi_4(\overrightarrow{n_1}),z)/j_0$ &  $2$  & $-1$ & $-1$   & $0$    & $2$    &  $2$   &  $0$   &  $0$   
\\
$\frac{f_{13,1}(z)}{f_{13,6}(z)}j(13,\pi_5(\overrightarrow{n_1}),z)/j_0$ &  $2$  & $-1$ & $-1$   & $1$    & $1$    &  $0$   &  $2$   &  $0$   
\\
$\frac{f_{13,1}(z)}{f_{13,6}(z)}j(13,\pi_6(\overrightarrow{n_1}),z)/j_0$ &  $1$  & $-1$ & $-1$   & $1$    & $1$    &   $1$  &  $1$   &  $1$   
\\  
$\frac{f_{13,1}(z)}{f_{13,6}(z)}\left(\frac{\eta(13z)}{\eta(z)}\right)^{-2}j(13,\pi_6(\overrightarrow{n_1}),z)/j_0$ &  $0$  & $0$ & $0$    &  $0$     & $0$   &   $0$     &  $0$   &  $0$              
\end{tabular}
\end{center}
where $2\le n\le 6$.
\\
\item[Step 4] Considering the LHS of equation \eqn{crank13id2} after division by $j_0$,
we now calculate lower bounds $\lambda(13,2,\zeta)$ of orders $ORD\left(\frac{\mathcal{S}_{13,2}^{}(\zeta_{13},z)}{j_0},\zeta,\Gamma_1(13)\right)$ at cusps $\zeta$ of $\Gamma_1(13)$.
\\
\begin{center}
\begin{tabular}{c c c c c}
cusp $\zeta$ & $\lambda(13,2,\zeta)$     \\
$i\infty$    &            \\
$0$          &   $-1$     \\
$1/n$        &   $0$      \\
$2/13$       &   $\CL{-14/13}=-1$ \\
$3/13$       &   $\CL{-11/13}=0$ \\
$4/13$       &   $\CL{1/13}=1$   \\
$5/13$       &   $\CL{-4/13}=0$  \\
$6/13$       &   $-1$     \\
\end{tabular} 
\end{center} 
where $2\le n\le 6$ using Theorem \thm{CKpmords}. We note that each value is an integer. 
\\
\item[Step 5] We summarize the calculations in Steps 4 and 5 in a Table.
The gives lower bounds for the LHS and RHS of equation \eqn{crank13id2} divided by $j_0$,
at the cusps $\zeta$.
\\
\begin{center}
\begin{tabular}{c c c c c}
cusp $\zeta$ & $n(\Gamma_1(13);\zeta)$ & $\ORD(LHS;\zeta)$ & $\ORD(RHS;\zeta)$ 
& $\ORD(LHS-RHS;\zeta)$ \\
$i\infty$    & $1$  &            &          &             \\
$0$          & $13$ &  $\ge-1$   &  $\ge-1$ &   $\ge -1$  \\
$1/n$        & $13$ &  $0$       &  $\ge-1$ &   $\ge -1$  \\
$2/13$       & $1$  &  $\ge-1$   &  $-2$    &   $\ge -2$ \\ 
$3/13$       & $1$  &  $\ge0$    &  $-1$    &   $\ge -1$ \\ 
$4/13$       & $1$  &  $\ge1$    &  $-1$    &   $\ge -1$  \\ 
$5/13$       & $1$  &  $\ge0$    &  $-1$    &   $\ge -1$ \\
$6/13$       & $1$  &  $\ge-1$   &  $-2$    &   $\ge -2$ \\
\end{tabular} 
\end{center} 
where $2\le n\le 6$.
\\\\
The constant $B$ is the sum of the lower bounds in the last column, so that $B=-9$.
\\
\item[Step 6] The LHS and RHS are weakly holomorphic modular forms of weight $0$ on $\Gamma_1(13)$. So in the Valence Formula, $\frac{\mu k}{12} = 0$. The result follows provided we can show that $\ORD(LHS-RHS,i\infty,\Gamma_1(13))\geq 14$. This is easily verified using MAPLE (see \href{https://github.com/rishabh-sarma/crank-identities-modulo-13-17-19}{\textit{https://github.com/rishabh-sarma/crank-identities-modulo-13-17-19}}).
\end{enumerate}

\subsubsection{\textbf{A quadratic non-residue case}}

The following is an identity for $\mathcal{S}_{13,1}^{}(\zeta_{13},z)$ in terms of generalized eta-functions : 
\\
\begin{align}
\mathcal{S}_{13,1}^{}(\zeta_{13},z)&=q^\frac{1}{13}(q^{13};q^{13})_\infty \sum_{n=1}^\infty \Lpar{\sum_{k=0}^{12} M(k,13,13n-6)\,\zeta_{13}^k}q^n \label{eq:crank13id3} \\
\nonumber
\\
&=\frac{f_{13,1}(z)}{f_{13,5}(z)}\sum_{k=-1}^{0}\sum_{r=1}^{6} \left(\dfrac{\eta(13z)}{\eta(z)}\right)^{2k}c_{13,r,k}\,j(13,\pi_r(\overrightarrow{n_1}),z),
\nonumber
\end{align}
where 
\begin{align*}
\overrightarrow{n_1}=(15,-1,-3,-2,-2,-2,-3),
\end{align*}
\\
and the coefficients $c_{13,r,k}$, $1\leq r \leq 6, -1\leq k \leq 0$ are linear combinations of cyclotomic integers like the identities found previously. We do not include them here, but the exact identity along with the maple code can be found here : \href{https://github.com/rishabh-sarma/crank-identities-modulo-13-17-19}{\textit{https://github.com/rishabh-sarma/crank-identities-modulo-13-17-19}}.
Utilizing the same algorithm, the proof follows similar to the identity for the quadratic residue case $\mathcal{S}_{13,2}^{}(\zeta_{13},z)$ done previously. We omit the details here.
\\\\
Below, we present one identity for each of the quadratic residue, quadratic non-residue and $m=0$ cases for $p=17$ and $p=19$. These are new and do not seem to appear in the literature elsewhere. We however make note of the fact that the same set of theta functions $j$ and vectors $\overrightarrow{n}$ are involved in the corresponding identity for the rank (\cite{Ga-Ri}, Subsections 7.3 and 7.4). However, there are much fewer functions in the corresponding identity for the crank here. We do not include the cyclotomic coefficients here, but the exact identities alongwith the maple code can be found here : \href{https://github.com/rishabh-sarma/crank-identities-modulo-13-17-19}{\textit{https://github.com/rishabh-sarma/crank-identities-modulo-13-17-19}}. Several of these coefficients are zero.

\subsection{Crank mod 17 identities}

\subsubsection{\textbf{Identity for $\mathcal{S}_{17,0}^{}$}}

The following is an identity for $\mathcal{S}_{17,0}^{}(\zeta_{17},z)$ in terms of generalized eta-functions :
\\
\begin{align}
\mathcal{S}_{17,0}^{}(\zeta_{17},z)&=(q^{17};q^{17})_\infty \sum_{n=1}^\infty \Lpar{\sum_{k=0}^{16} M(k,17,17n-12)\,\zeta_{17}^k}q^n \label{eq:crank17id1} \\
\nonumber
\\
\nonumber
&=\sum_{k=0}^{1}\sum_{r=1}^{8} \left(\dfrac{\eta(17z)}{\eta(z)}\right)^{3k}\left(c_{17,r,k}\,j(17,\pi_r(\overrightarrow{n_1}),z)+
d_{17,r,k}\,j(17,\pi_r(\overrightarrow{n_2}),z)\right),
\nonumber
\end{align}
\\
where 
\begin{align*}
\overrightarrow{n_1}=(15,-3,-1,-2,-1,-2,-1,-2,-1)
\\
\overrightarrow{n_2}=(27,-2,-2,-3,-2,-4,-4,-4,-4),
\end{align*}
and the coefficients $c_{17,r,k}, d_{17,r,k}$, $1\leq r \leq 8, 0\leq k \leq 1$ are linear combinations of cyclotomic integers like the mod $11$ and $13$ identities found previously. 

\subsubsection{\textbf{A quadratic residue case}}

The following is an identity for $\mathcal{S}_{17,12}^{}(\zeta_{17},z)$ in terms of generalized eta-functions :
\\
\begin{align}
\mathcal{S}_{17,12}^{}(\zeta_{17},z)&=q^\frac{12}{17}(q^{17};q^{17})_\infty \Bigg(\sum_{n=1}^\infty \Lpar{\sum_{k=0}^{16} M(k,17,17n)\,\zeta_{17}^k}q^n \Bigg)\label{eq:crank17id2} \\
\nonumber
\\
&=\frac{f_{17,7}(z)}{f_{17,5}(z)}\sum_{k=-1}^{1}\sum_{r=1}^{8} \left(\dfrac{\eta(17z)}{\eta(z)}\right)^{3k}\left(c_{17,r,k}\,j(17,\pi_r(\overrightarrow{n_1}),z)\,+
d_{17,r,k}\,j(17,\pi_r(\overrightarrow{n_2}),z)\right),
\nonumber
\\
\nonumber
\end{align}
where
\begin{align*}
&\overrightarrow{n_1}=(15,-3,-1,-2,-1,-2,-1,-2,-1),
\\
&\overrightarrow{n_2}=(27,-2,-2,-3,-2,-4,-4,-4,-4),
\end{align*}
and the coefficients $c_{17,r,k}, d_{17,r,k}$, $1\leq r \leq 8, -1\leq k \leq 1$ are linear combinations of cyclotomic integers like the identities found previously. 

\subsubsection{\textbf{A quadratic non-residue case}}
The following is an identity for $\mathcal{S}_{17,1}^{}(\zeta_{17},z)$ in terms of generalized eta-functions :
\\
\begin{align}
\mathcal{S}_{17,1}^{}(\zeta_{17},z)&=q^{\frac{1}{17}}(q^{17};q^{17})_\infty \sum_{n=1}^\infty \Lpar{\sum_{k=0}^{16} M(k,17,17n-11)\,\zeta_{17}^k}q^n \label{eq:crank17id3} \\
\nonumber
\\
&=\frac{f_{17,7}(z)}{f_{17,8}(z)}\sum_{k=-1}^{1}\sum_{r=1}^{8} \left(\dfrac{\eta(17z)}{\eta(z)}\right)^{3k}\left(c_{17,r,k}\,j(17,\pi_r(\overrightarrow{n_1}),z)
+ d_{17,r,k}\,j(17,\pi_r(\overrightarrow{n_2}),z)\right),
\nonumber
\end{align}
where
\begin{align*}
&\overrightarrow{n_1}=(15,-3,-1,-2,-1,-2,-1,-2,-1),
\\
&\overrightarrow{n_2}=(27,-2,-2,-3,-2,-4,-4,-4,-4),
\end{align*}
and the coefficients $c_{17,r,k}, d_{17,r,k}$, $1\leq r \leq 8, -1\leq k \leq 1$ are linear combinations of cyclotomic integers like the identities found previously. 

\subsection{Crank mod 19 identities}

\subsubsection{\textbf{Identity for $\mathcal{S}_{19,0}^{}$}}

The following is an identity for $\mathcal{S}_{19,0}^{}(\zeta_{19},z)$ in terms of generalized eta-functions :
\\
\begin{align}
\mathcal{S}_{19,0}^{}(\zeta_{19},z)&=(q^{19};q^{19})_\infty \sum_{n=1}^\infty \Lpar{\sum_{k=0}^{18} M(k,19,19n-15)\,\zeta_{19}^k}q^n \label{eq:crank19id1} \\
\nonumber
\\
\nonumber
&=\sum_{k=0}^{1}\sum_{r=1}^{9} \left(\dfrac{\eta(19z)}{\eta(z)}\right)^{4k}\Big(c_{19,r,k}\,j(19,\pi_r(\overrightarrow{n_1}),z)+d_{19,r,k}\,j(19,\pi_r(\overrightarrow{n_2}),z)+
\\
\nonumber
&\hspace{45mm}e_{19,r,k}\,j(19,\pi_r(\overrightarrow{n_3}),z)\Big),
\end{align}
\\
where 
\begin{align*}
\overrightarrow{n_1}=(27,-3,-2,-4,-4,-3,-3,-2,-3,-1),
\\
\overrightarrow{n_2}=(39,-5,-2,-5,-5,-3,-5,-2,-5,-5),
\\
\overrightarrow{n_3}=(39,-5,-4,-3,-4,-5,-4,-4,-5,-3),
\end{align*}
and the coefficients $c_{19,r,k}, d_{19,r,k}$, and $e_{19,r,k}$ for $1\leq r \leq 9, 0\leq k \leq 1$ are linear combinations of cyclotomic integers like the mod $11, 13$ and $17$ identities found previously. 

\subsubsection{\textbf{A quadratic residue case}}

The following is an identity for $\mathcal{S}_{19,15}^{}(\zeta_{19},z)$ in terms of generalized eta-functions :
\\
\begin{align}
\mathcal{S}_{19,15}^{}(\zeta_{19},z)&=q^\frac{15}{19}(q^{19};q^{19})_\infty \Bigg(\sum_{n=1}^\infty \Lpar{\sum_{k=0}^{18} M(k,19,19n)\,\zeta_{19}^k}q^n\Bigg) \label{eq:crank19id2}\\
\nonumber
\\
&=\frac{f_{19,6}(z)}{f_{19,5}(z)}\Bigg(\sum_{k=-1}^{1}\sum_{r=1}^{9} \left(\dfrac{\eta(19z)}{\eta(z)}\right)^{4k}\Big(c_{19,r,k}\,j(19,\pi_r(\overrightarrow{n_1}),z)+d_{19,r,k}\,j(19,\pi_r(\overrightarrow{n_2}),z)
\nonumber
\\
\nonumber
&\hspace{63mm}+e_{19,r,k}\,j(19,\pi_r(\overrightarrow{n_3}),z)\Big)\Bigg),
\nonumber
\\
\nonumber
\end{align}
where
\begin{align*}
\overrightarrow{n_1}=(39, -5, -5, -4, -5, -4, -5, -3, -5, -1),
\\
\overrightarrow{n_2}=(39, -3, -5, -5, -5, -3, -2, -4, -5, -5),
\\
\overrightarrow{n_3}=(39, -4, -5, -4, -3, -3, -4, -5, -5, -4),
\end{align*}
and the coefficients $c_{19,r,k}, d_{19,r,k}, e_{19,r,k}$, $1\leq r \leq 9, -1\leq k \leq 1$ are linear combinations of cyclotomic integers like the identities found previously. 

\subsubsection{\textbf{A quadratic non-residue case}}
The following is an identity for $\mathcal{S}_{19,1}^{}(\zeta_{19},z)$ in terms of generalized eta-functions :
\\
\begin{align}
\mathcal{S}_{19,1}^{}(\zeta_{19},z)&=q^{\frac{1}{19}}(q^{19};q^{19})_\infty \sum_{n=1}^\infty \Lpar{\sum_{k=0}^{18} M(k,19,19n-14)\,\zeta_{19}^k}q^n \label{eq:crank19id3} \\
\nonumber
\\
&=\frac{f_{19,8}(z)}{f_{19,9}(z)}\Bigg(\sum_{k=-1}^{1}\sum_{r=1}^{9} \left(\dfrac{\eta(19z)}{\eta(z)}\right)^{4k}\Big(c_{19,r,k}\,j(19,\pi_r(\overrightarrow{n_1}),z)+d_{19,r,k}\,j(19,\pi_r(\overrightarrow{n_2}),z)
\nonumber
\\
\nonumber
&\hspace{63mm}+e_{19,r,k}\,j(19,\pi_r(\overrightarrow{n_3}),z)\Big)\Bigg),
\nonumber
\nonumber
\end{align}
where
\begin{align*}
\overrightarrow{n_1}=(39, -5, -5, -4, -5, -4, -5, -3, -5, -1),
\\
\overrightarrow{n_2}=(39, -3, -5, -5, -5, -3, -2, -4, -5, -5),
\\
\overrightarrow{n_3}=(39, -4, -5, -4, -3, -3, -4, -5, -5, -4),
\end{align*}
and the coefficients $c_{19,r,k}, d_{19,r,k}, e_{19,r,k}$, $1\leq r \leq 9, -1\leq k \leq 1$ are linear combinations of cyclotomic integers like the identities found previously.

\section{Concluding remarks}
\label{sec:conclusion}

Building on our work on symmetries of the rank generating function \cite{Ga-Ri}, we have found analogous results for the crank generating function in this paper. The transformation and modularity properties as well as the $p$-dissection identities are similar for the rank and crank functions. It would be interesting to examine the underlying theory behind this resemblance more closely. Two other natural problems to consider are finding crank dissection identities for higher primes $p\geq23$ and also exploring such symmetries for other partition statistics and also their overpartition analogues.

\section{Acknowledgments}
\label{sec:acknowledgements}
The author would like to thank his advisor Frank Garvan for suggesting this problem and his support and guidance throughout the course of this project.

\bibliographystyle{amsplain}

\begin{thebibliography}{20}


\bibitem{An-Ga88}
G.E.~Andrews and F.~G.~Garvan,                   
\textit{Dyson's crank of a partition},
Bull. Amer. Math. Soc. (N.S.)
\textbf{18} 
(1988),
167--171.
\bibitem{Ap76}
T.M.~Apostol,                   
\textit{Modular functions and Dirichlet series in number theory},
{Graduate Texts in Mathematics}, No. 41, Springer-Verlag, New York-Heidelberg, 1976.
\bibitem{At-Hu}
A.~O.~L. Atkin and S.~M. Hussain, \emph{Some properties of partitions. {II}},
  Trans. Amer. Math. Soc. \textbf{89} (1958), 184--200. \MR{103872}
\bibitem{At-Sw}
A.~O.~L. Atkin and  P.~Swinnerton-Dyer,
\textit{Some properties of partitions},
Proc. London Math. Soc. (3)
\textbf{4} 
(1954),
84--106.
\bibitem{Bi89}
A.~J.~F.~Biagioli, 
\textit{A proof of some identities of {R}amanujan using modular forms},
Glasgow Math. J.
\textbf{31} 
(1989),
271--295.
\bibitem{Bi13}
G.~Bilgici, 
\textit{The crank of partitions mod 13},
Ramanujan J.
\textbf{30}(3) 
(2013),
403--424.
\bibitem{Br-On10}
K.~Bringmann and K.~Ono,
\textit{Dyson's ranks and {M}aass forms},
Ann. of Math. (2)
\textbf{171} 
(2010),
419--449.
\bibitem{Br-Sw}
K.~Bringmann and H.~Swisher,
\textit{On a conjecture of {K}oike on identities between {T}hompson
series and {R}ogers-{R}amanujan functions},
Proc. Amer. Math. Soc.
\textbf{135} 
(2007),
2317--2326.
\bibitem{Ch-Ko-Pa}
B.~Cho, J.~K.~Koo, Y.~K.~Park,
\textit{Arithmetic of the Ramanujan-G\"ollnitz-Gordon continued fraction},
J. Number Theory
\textbf{129} 
(2009),
922--947.
\bibitem{Dy44}
F.~J.~Dyson,
\textit{Some guesses in the theory of partitions},
Eureka (Cambridge)
\textbf{8} 
(1944),
10--15.
\bibitem{Ga90}
F.~G.~Garvan,
\textit{The crank of partitions mod $8$,$9$ and $10$},
Trans. Amer. Math. Soc.
\textbf{322}
(1990),
79--94.
\bibitem{Ga19}
F.~G.~Garvan,
\textit{Transformation properties for Dyson's rank function},
Trans. Amer. Math. Soc.
\textbf{371}
(2019),
199--248.
\bibitem{Ga-Ki-St}
F.~G.~Garvan, D.~Kim and D.~W. Stanton, 
\textit{Cranks and $t$-cores}, 
Invent. Math. 
{\bf 101} 
(1990), 
no.~1, 
1--17.
\bibitem{Ga-Ri}
F.~G.~Garvan, and R.~Sarma,
\textit{New symmetries for Dyson's rank function}, 
Ramanujan J.
\textbf{65} 
(2024), 
1883--1939.
\bibitem{Hi90}
M.~D.~Hirschhorn,
\textit{A birthday present for Ramanujan},
Amer. Math. Monthly
\textbf{97}
(1990),
398--400.
\bibitem{Kn-book}
M.~I.~Knopp,
``Modular Functions in Analytic Number Theory,''
Markham Publishing Co., Chicago, Illinois, 1970.
\bibitem{Mah}
K.~Mahlburg,
\textit{Partition congruences and the Andrews-Garvan-Dyson crank},
Proc. Natl. Acad. Sci. USA
\textbf{102} 
(2005),
15373--15376.
\bibitem{Ko-book}
G.~K{\"o}hler,
\textit{Eta Products and Theta Series Identities},
{Springer Monographs in Mathematics}, Springer Heidelberg, 2011.
\bibitem{OB-thesis}
John~Noel O'Brien, \emph{Some properties of partitions with special reference
  to primes greater then $5$, $7$ and $11$}, Ph.D. thesis, Univ. of Durham,
  England, 1965, 107pp.
\bibitem{Ra}
R.~A.~Rankin,
``Modular Forms and Functions,''
Cambridge University Press, 1977.
\bibitem{Yang04}
Y.~Yang, 
\textit{Transformation formulas for generalized Dedekind eta functions},
Bull. London Math. Soc.
\textbf{36} 
(2004),
671--682.
\bibitem{Zw00}
S.~P.~Zwegers, 
\textit{Mock {$\theta$}-functions and real analytic modular forms},
in
``$q$-Series with Applications to Combinatorics, Number
Theory, and Physics,''
Contemp. Math.
\textbf{291} 
(2001),
269--277.
\bibitem{Zw-thesis}          
S.~P.~Zwegers, 
``Mock Theta Functions,''
Ph.D. thesis, Universiteit Utrecht, 2002, 96 pp.



\end{thebibliography}

\end{document}